\documentclass[12pt, reqno]{amsart}
\usepackage{amssymb, amsthm, amsmath, amsfonts}
\usepackage{array, epsfig}
\usepackage{bbm}
\usepackage{hyperref}
\usepackage{color}

  \evensidemargin
\oddsidemargin
\newcommand{\rank}{\mathop{\mathrm{rank}}\nolimits}
\newcommand{\codim}{\mathop{\mathrm{codim}}\nolimits}

\newcommand{\cA}{\mathcal{A}}
\newcommand{\cB}{\mathcal{B}}
\newcommand{\cC}{\mathcal{C}}

\newcommand{\cF}{\mathcal{F}}

\newcommand{\cL}{\mathcal{L}}

\newcommand{\cR}{\mathcal{R}}

\newcommand{\Sym}{\text{\rm Sym}}

\DeclareMathOperator*{\Ker}{Ker}

\newcommand{\E}{\mathbb E}

\newcommand{\R}{\mathbb{R}}
\newcommand{\N}{\mathbb{N}}

\newcommand{\Z}{\mathbb{Z}}

\renewcommand{\P}{\mathbb{P}}

\newcommand{\Int}{\mathop{\mathrm{Int}}\nolimits}

\newcommand{\Cone}{\mathop{\mathrm{Cone}}\nolimits}

\newcommand{\conv}{\mathop{\mathrm{Conv}}\nolimits}
\newcommand{\aff}{\mathop{\mathrm{aff}}\nolimits}
\newcommand{\lin}{\mathop{\mathrm{lin}}\nolimits}

\newcommand{\eps}{\varepsilon}
\newcommand{\eqdistr}{\stackrel{d}{=}}

\newcommand{\bsl}{\backslash}
\newcommand{\ind}{\mathbbm{1}}

\newcommand{\dd}{{\rm d}}
\newcommand{\eee}{{\rm e}}

\newcommand{\bsigma}{{\boldsymbol{\varsigma}}}

\theoremstyle{plain}
\newtheorem{theorem}{Theorem}[section]
\newtheorem{lemma}[theorem]{Lemma}

\newtheorem{proposition}[theorem]{Proposition}

\theoremstyle{definition}

\newtheorem{example}[theorem]{Example}

\theoremstyle{remark}
\newtheorem{remark}[theorem]{Remark}

\newcommand{\stirling}[2]{\genfrac{[}{]}{0pt}{}{#1}{#2}}
\newcommand{\stirlingsec}[2]{\genfrac{\{}{\}}{0pt}{}{#1}{#2}}

\begin{document}

\author{Zakhar Kabluchko}
\address{Zakhar Kabluchko: Institut f\"ur Mathematische Stochastik,
Westf\"alische Wilhelms-Universit\"at M\"unster,
Orl\'eans--Ring 10,
48149 M\"unster, Germany}
\email{zakhar.kabluchko@uni-muenster.de}

\author{Vladislav Vysotsky}
\address{Vladislav Vysotsky, University of Sussex,
Pevensey 2 Building,
Falmer Campus,
Brighton BN1 9QH,
United Kingdom
and
St.\ Petersburg Department of Steklov Mathematical Institute,
Fontanka~27,
191011 St.\ Petersburg,
Russia}
\email{v.vysotskiy@sussex.ac.uk, vysotsky@pdmi.ras.ru}

\author{Dmitry Zaporozhets}
\address{Dmitry Zaporozhets, St.\ Petersburg Department of Steklov Mathematical Institute,
Fontanka~27,
191011 St.\ Petersburg,
Russia}
\email{zap1979@gmail.com}

\title[Convex hulls of random walks]{Convex hulls of random walks: Expected number of faces and face probabilities}
\thanks{This paper was written when V.V. was affiliated to Imperial College London, where his work was supported by People Programme (Marie Curie Actions) of the European Union's Seventh Framework Programme (FP7/2007-2013) under REA grant agreement n$^\circ$[628803].
His work is also supported in part by Grant 16-01-00367 by RFBR.
The work of D.Z.\ is supported in parts by Grant 16-01-00367 by RFBR, the Program of Fundamental Researches of Russian Academy of Sciences ``Modern Problems of Fundamental Mathematics'', and by Project SFB 1283 of Bielefeld University}


\begin{abstract}
Consider a sequence of partial sums $S_i= \xi_1+\dots+\xi_i$, $1\leq i\leq n$, starting at $S_0=0$, whose increments $\xi_1,\dots,\xi_n$ are random vectors in $\mathbb R^d$, $d\leq n$.  We are interested in the properties of the convex hull $C_n:=\mathrm{Conv}(S_0,S_1,\dots,S_n)$.
Assuming that the tuple $(\xi_1,\dots,\xi_n)$ is exchangeable
and a certain general position condition holds, we prove that the expected number of $k$-dimensional faces of $C_n$ is given by the formula
$$
\E [f_k(C_n)] = \frac{2\cdot k!}{n!} \sum_{l=0}^{\infty}\genfrac{[}{]}{0pt}{}{n+1}{d-2l}  \genfrac{\{}{\}}{0pt}{}{d-2l}{k+1},
$$
for all $0\leq k \leq d-1$, where $\genfrac{[}{]}{0pt}{}{n}{m}$ and $\genfrac{\{}{\}}{0pt}{}{n}{m}$ are Stirling numbers of the first and second kind, respectively.

Further, we compute explicitly the probability that for given indices $0\leq i_1<\dots <i_{k+1}\leq n$, the points $S_{i_1},\dots,S_{i_{k+1}}$ form a $k$-dimensional face of $\mathrm{Conv}(S_0,S_1,\dots,S_n)$. This is done in two different settings: for random walks with symmetrically exchangeable increments and for random bridges with exchangeable increments. These results generalize the classical one-dimensional discrete arcsine law for the position of the maximum due to E.\ Sparre Andersen.
All our formulae are distribution-free, that is do not depend on the distribution of the increments~$\xi_k$'s.

The main ingredient in the proof  is the computation of the probability that the origin is absorbed by a {\it joint} convex hull of several random walks and bridges whose increments are invariant with respect to the action of direct product of finitely many reflection groups of types $A_{n-1}$ and $B_n$. This probability, in turn, is related to the number of Weyl chambers of a product-type reflection group that are intersected by a linear subspace in general position.
\end{abstract}

\maketitle

\section{Statement of main results}\label{1306}

\subsection{Introduction}
Let $\xi_1,\dots,\xi_n$ be (possibly dependent) random $d$-dimensional vectors with partial sums
$$
S_i = \xi_1 + \dots + \xi_i,\quad  1\leq i\leq n,\quad  S_0=0.
$$
The sequence $S_0,S_1,\dots,S_n$ will be referred to as \emph{random walk} or, if the additional boundary condition $S_n=0$ is imposed, a \emph{random bridge}.

In the one-dimensional case $d=1$, Sparre Andersen~\cite{Sparre2, sparre_andersen1,sparre_andersen2} derived remarkable formulae for several functionals of the random walk $S_0,S_1,\dots,S_n$ including the number of positive terms and the position of the maximum.
More specifically, assuming that the joint distribution of the increments $(\xi_1,\dots,\xi_n)$ is invariant under arbitrary  permutations and sign changes and that $\P[S_i=0]=0$ for all $1\leq i\leq n$, Sparre Andersen proved in~\cite[Theorem~C]{sparre_andersen2} the following \emph{discrete arcsine law} for the position of the maximum:
\begin{equation}\label{eq:arcsine_maximum_0}
\P\left[\max\{S_0,\dots,S_n\} = S_i\right] =
\frac 1 {2^{2n}} \binom{2i}{i} \binom{2n-2i}{n-i},
\quad
i=0,\dots,n.
\end{equation}
By the symmetry, the same holds for the position of the minimum. Surprisingly, the above formula is distribution-free, that is its right-hand side does not depend on the distribution of $(\xi_1,\dots,\xi_n)$ provided the symmetric exchangeability and the general position assumptions mentioned above are satisfied. Another unexpected consequence of this formula is that the maximum is more likely to be attained at $i=0$ or $i=n$ rather than at $i\approx n/2$, as one could na\"{i}vely guess. A discussion of the arcsine laws can be found in Feller's book~\cite[Vol II, Section XII.8]{Feller}.

Let us now turn to the general $d$-dimensional case and ask ourselves what could be an appropriate multidimensional generalization of~\eqref{eq:arcsine_maximum_0}. Of course, the maximum and the minimum are not well defined for multidimensional random walks, but instead we can consider vertices (and, more generally, faces) of the \emph{convex hull} 
\begin{equation*}
C_n := \conv(S_0,S_1,\dots, S_n)
=\{\alpha_0 S_0+\dots+\alpha_n S_n \colon \alpha_0,\dots,\alpha_n \geq 0, \alpha_0+\dots+\alpha_n =1\}.
\end{equation*}
Clearly, $C_n$ is a random polytope in $\R^d$ whose vertices belong to the collection $\{S_0,S_1,\dots,S_n\}$. In the one-dimensional case, $C_n$ has two vertices, the maximum and the minimum, but in higher dimensions the question on the number of vertices (or, more generally, faces) and their positions becomes non-trivial.
The main results of the present paper can be summarized as follows.
Under the appropriate exchangeability and general position assumptions on the random walk or bridge, we compute
\begin{enumerate}
\item [(a)] 
the expected number of $k$-dimensional faces of $C_n$, for all $0\leq k\leq d-1$, and
\item [(b)] 
the probability that the simplex $\conv(S_{i_1},\dots,S_{i_{k+1}})$ is a $k$-dimensional face of the convex hull $C_n$, for a given collection of indices $0\leq i_1 < \dots < i_{k+1}\leq n$.
\end{enumerate}

All formulae turn out to be distribution-free. The probabilities in (b), referred to as \emph{face probabilities}, will be interpreted below in terms of the so-called \emph{absorption probabilities}, that is the probabilities that a joint convex hull of several random walks and random bridges contains the origin. In our recent work~\cite{KVZ15}, we showed that in the case of just \emph{one} random walk or bridge, the absorption probability can be computed in a purely geometric way by counting the number of Weyl chambers of a certain reflection group that are intersected  by a linear subspace in general position. Moreover, we showed in~\cite{KVZ15} that random walks correspond to Weyl chambers of type $B_n$, whereas random bridges correspond to type $A_{n-1}$ chambers.
In the present paper, we extend the results of~\cite{KVZ15} to \emph{joint} convex hulls of \emph{several} random walks and bridges. We shall show that the corresponding absorption probabilities can be interpreted in terms of Weyl chambers of \emph{product type}. Our main  formula for the absorption probabilities will be stated in Theorem~\ref{1435}, below.

We shall argue below that (b) can be viewed as a generalization of the discrete arcsine law to higher dimensions. Let us mention that there is another discrete arcsine law, also due to Sparre Andersen~\cite[Theorem~C]{sparre_andersen2},  for the number of positive terms in a random walk. A multidimensional generalization of this result is considered in our separate paper~\cite{KVZ16_arcsine}.

Convex hulls of random walks and their applications to Brownian motions and L\'evy processes were much studied; see, e.g.,\ \cite{Nielsen,Baxter,Eldan0,kabluchko_zaporozhets_sobolev,kampf_etal,MCR10,molchanov_wespi,SS,SW,vysotsky_zaporozhets,wade_xu}.  These papers concentrate mostly on functionals like the volume and the perimeter, which are not distribution-free. Face probabilities for faces of maximal dimension were computed by Barndorff-Nielsen and Baxter~\cite{Nielsen} and Vysotsky and Zaporozhets~\cite{vysotsky_zaporozhets}. We shall recover the corresponding formula as a special case of our results, but our methods are different from that of~\cite{Nielsen, vysotsky_zaporozhets}.
Reviews of the literature on random convex hulls and random polytopes can be found in~\cite{hug_survey,MCR10,schneider_polytopes}.

\subsection{Expected number of \texorpdfstring{$k$}{k}-faces}
Recall that $C_n= \conv(S_0,\dots,S_n)$ denotes the convex hull of a $d$-dimensional random walk $(S_i)_{i=0}^n$ with increments $\xi_1,\dots,\xi_n$. The increments are random vectors which may be dependent. Our first result is a formula for the expected number of $k$-dimensional faces  of $C_n$.
To state it, we need to impose the following assumptions on the joint distribution of  the increments. 
\begin{enumerate}
\item[$(\text{Ex})$] \textit{Exchangeability:} For every permutation $\sigma$ of the set $\{1,\dots,n\}$, we have the distributional equality
$$
(\xi_{\sigma(1)},\dots,  \xi_{\sigma(n)}) \eqdistr (\xi_1,\dots,\xi_n).
$$
\item[$(\text{GP})$] \textit{General position:}
For every $1\leq i_1 < \dots < i_d\leq n$, the probability that the vectors $S_{i_1}, \dots,S_{i_d}$ are linearly dependent is $0$.
\end{enumerate}

\begin{example}
Conditions $(\text{Ex})$ and $(\text{GP})$ are satisfied if $\xi_1,\dots,\xi_n$ are independent identically distributed, and for every hyperplane $H_0\subset \R^d$ passing through the origin we have $\P[S_i\in H_0] = 0$, for all $1\leq i\leq n$. The proof of $(\text{GP})$ can be found in~\cite[Proposition~2.5]{KVZ15}, while the proof of $(\text{Ex})$ is trivial. The assumption $\P[S_i\in H_0] = 0$, $1\leq i\leq n$, in turn follows (for i.i.d. increments) if we assume that $\P[\xi_1 \in H]=0$ for every affine hyperplane $H \subset \R^d$; this is again shown in the proof of~\cite[Proposition~2.5]{KVZ15}. To clarify, the proposition additionally assumes (although does not state explicitly in the published version)  that $\xi_1 \eqdistr - \xi_1$, but this is not required for the implications mentioned above.
\end{example}

Denote by $\mathcal F_k(C)$, where $0\leq k\leq d-1$, the set of all $k$-dimensional faces (or just $k$-faces, in short) of a convex polytope $C$. Let $f_k(C)$ be the number of $k$-faces of $C$:
\[
f_k(C):=\#\mathcal F_k(C).
\]
Note that under assumption $(\text{GP})$, all faces of $C_n$ are simplices with probability $1$; see Remark~\ref{rem:simplicial}, below for a proof.

\begin{theorem}\label{theo:expected_walk}
Let $(S_i)_{i=0}^n$ be a random walk in $\R^d$, $n\geq d$,  whose increments $\xi_1,\dots,\xi_n$ satisfy conditions $(\text{Ex})$ and $(\text{GP})$.
Then, for all $0\leq k\leq d-1$, 
\begin{equation}\label{eq:E_F_k_C_n_main_theorem}
\E [f_k(C_n)] = \frac{2\cdot k!}{n!} \sum_{l=0}^{\infty}\stirling{n+1}{d-2l}  \stirlingsec{d-2l}{k+1}.
\end{equation}
\end{theorem}
 The right-hand side  contains the (signless) \emph{Stirling numbers of the first kind} $\stirling{n}{m}$ and the \emph{Stirling numbers of the second kind} $\stirlingsec{n}{m}$, where $m, n\in\N$ and $1\leq m \leq n$, which are defined as the number of permutations of an $n$-element set with exactly $m$ cycles and the number of partitions of an $n$-element set into $m$ non-empty subsets, respectively. The exponential generating functions of the Stirling numbers are given by
\begin{equation}\label{eq:stirling_def}
\sum_{n=m}^{\infty} \stirling{n}{m}\frac{t^n}{n!} = \frac 1 {m!} \left(\log \frac 1 {1-t}\right)^m,
\quad
\sum_{n=m}^{\infty} \stirlingsec{n}{m}\frac{t^n}{n!} = \frac 1 {m!} (\eee^{t}-1)^m.
\end{equation}
For these and other properties of Stirling numbers, we refer to~\cite[Chapters~6 and~7]{graham_knuth_patashnik_book}. For $n\in\N$,  $m\in \Z\backslash \{1,\dots,n\}$ and $n \notin \N$ we use the convention $\stirling{n}{m} = \stirlingsec{n}{m}=0$, so that the sum in~\eqref{eq:E_F_k_C_n_main_theorem} contains only finitely many non-vanishing terms.
The Stirling numbers of the first kind can also be defined as the coefficients of the rising factorial
\begin{equation}\label{eq:rising_factorial}
t^{(n)} := t(t+1)(t+2)\dots (t+n-1) = \sum_{j=0}^n \stirling{n}{j} t^j.
\end{equation}

\begin{remark}
Let us mention some special cases of Theorem~\ref{theo:expected_walk}. For faces of maximal dimension (where $k=d-1$ and only the term with $l=0$ is present) and vertices (where $k=0$), formula~\eqref{eq:E_F_k_C_n_main_theorem} simplifies to
$$
\E\, [f_{d-1}(C_n)]
=
\frac {2(d-1)!} {n!}
\stirling{n+1}{d},
\quad
\E\, [f_{0}(C_n)]
=
\frac 2 {n!} \sum_{l=0}^\infty
 \stirling{n+1}{d-2l},
$$
where we used the identities $\stirlingsec{d}{d} = 1$ and $\stirlingsec{d-2l}{1} = 1$ (for $d-2l\geq 1$).  For example, in dimension $d=2$, both the expected number of edges and the expected number of vertices of the random polygon $C_n$ are equal to $2H_{n} := 2 \left(1+\frac 12 +\dots + \frac 1{n}\right)$ since $\stirling {n+1}2 = n! H_{n}$. This result was known~\cite{Baxter, vysotsky_zaporozhets}. In dimension $d=1$ we recover the trivial formula $\E\, [f_{0}(C_n)] = 2$, which accounts the two vertices being the maximum and the minimum of the random walk, since $\stirling {n+1}1 = n!$.
\end{remark}

\begin{remark}
Using the fixed $k$ asymptotic formula for Stirling numbers of the first kind, see~\cite[page~160]{jordan_book} (or~\cite{wilf} for much more precise asymptotics), namely
\begin{equation}\label{eq:stirling_asympt}
\frac 1{(n-1)!} \stirling{n}{k} \sim \frac{(\log n)^{k-1}}{(k-1)!}, \quad n\to\infty, \quad k \text{ fixed},
\end{equation}
we obtain 
\begin{equation}\label{eq:F_k_asympt}
\E\, [f_k(C_n)]  \sim \frac{2\cdot k!}{(d-1)!} \stirlingsec{d}{k+1} (\log n)^{d-1}, \quad n\to\infty, \quad k, d \text{ fixed}.
\end{equation}
Here, $a_n\sim b_n$ means that $\lim_{n\to\infty} a_n/b_n = 1$. Interestingly, the same asymptotics (up to a constant factor) holds for the expected number of $k$-faces of the convex hull of $n$ i.i.d.\ points uniformly distributed in a $d$-dimensional convex polytope; see~\cite{mR05}.
\end{remark}


\begin{remark}\label{rem:simplicial}
Under assumptions $(\text{Ex})$ and $(\text{GP})$, each face $g\in\mathcal F_k(C_n)$ of the polytope $C_n$ is, with probability $1$, a $k$-dimensional simplex of the form
$$
g = \conv(S_{j_1(g)},\dots,S_{j_{k+1}(g)})
$$
for some indices $0\leq j_1(g)<\dots<j_{k+1}(g)\leq n$. It suffices to prove this for $k=d-1$ because all faces of a simplex are simplices. For all $0\leq i_1<\dots<i_{d+1}\leq n$ we have
\begin{multline*}
\P[S_{i_1},\dots,S_{i_{d+1}} \text{ are contained in a common hyperplane}]
\\
\begin{aligned}
&=
\P[S_{i_2}-S_{i_1},S_{i_3}-S_{i_1},\dots,S_{i_{d+1}}-S_{i_1} \text{ are linearly dependent}] \\
&=
\P[S_{i_2-i_1},S_{i_3-i_1}, \dots,S_{i_{d+1}-i_1} \text{ are linearly dependent}] \\
&=
0
\end{aligned}
\end{multline*}
by assumptions $(\text{Ex})$ and $(\text{GP})$. It follows that, with probability $1$, every $(d-1)$-dimensional face of $C_n$ contains at most $d$ vertices and, consequently, is a simplex.
\end{remark}

\subsection{Face probabilities for symmetric random walks}
In the next theorem we compute the probability that a given collection of points of the random walk forms a face of the convex polytope $C_n$. To state it, we need an assumption which, in addition to exchangeability, requires invariance with respect to sign changes:

\begin{itemize}
\item [$(\pm\text{Ex})$] \textit{Symmetric exchangeability:} For every permutation $\sigma$ of the set $\{1,\dots,n\}$ and every $\eps_1,\dots,\eps_n\in \{-1,+1\}$, there is the distributional equality
    $$
    (\xi_1,\dots,\xi_n) \eqdistr (\eps_1 \xi_{\sigma(1)}, \dots, \eps_n \xi_{\sigma(n)}).
    $$
\end{itemize}
Random walks satisfying $(\pm\text{Ex})$ will be frequently referred to as \emph{symmetric}. For example, $(\pm\text{Ex})$ is satisfied if $\xi_1,\dots,\xi_n$ are i.i.d.\ random vectors in $\R^d$ with centrally symmetric distribution (meaning that $\xi_1$ has the same distribution as $-\xi_1$).
\begin{theorem}\label{1249}
Let $(S_i)_{i=0}^n$ be a random walk in $\R^d$ whose increments $\xi_1,\dots,\xi_n$ satisfy assumptions $(\pm\text{Ex})$ and $(\text{GP})$.
Fix some $0\leq k\leq d-1$ and let $0 \le i_1 < \dots < i_{k+1} \le n$ be any indices. Then,
$$
\P[\conv(S_{i_1},\dots,S_{i_{k+1}}) \in \mathcal F_k(C_n)]
=
\frac{2(P_{i_1,\dots,i_{k+1}}^{(n)}(d-k-1) + P_{i_1,\dots,i_{k+1}}^{(n)}(d-k-3)+\dots)}
{2^{i_1+n-i_{k+1}} i_1!(i_2-i_1)!\dots  (i_{k+1}-i_k)! (n-i_{k+1})!} ,
$$
where the $P_{i_1,\dots,i_{k+1}}^{(n)}(j)$'s  are the coefficients of the polynomial
\begin{multline*}
(t+1)(t+3)\dots (t+2i_1-1)\times(t+1)(t+3)\dots (t+2(n-i_{k+1})-1)\\
\times\prod_{l=1}^k ((t+1)(t+2)\dots (t+i_{l+1}-i_l-1)) = \sum_{j=0}^{n-k} P_{i_1,\dots,i_{k+1}}^{(n)}(j) t^j.
\end{multline*}
For $j<0$ or $j>n-k$ we use the convention $P_{i_1,\dots,i_{k+1}}^{(n)}(j) := 0$.
\end{theorem}

\begin{remark}
Take some $0\leq i\leq n$. For the probability that $S_i$ is a vertex of the convex hull $C_n$, we obtain, by taking $k=0$ in Theorem~\ref{1249},
\begin{equation}\label{eq:S_i_is_vertex}
\P[S_{i} \in \mathcal F_{0}(C_n)] =
\frac{P_{i}^{(n)}(d-1) + P_{i}^{(n)}(d-3)+\dots}{2^{n-1} i!(n-i)!},
\end{equation}
where the $P_{i}^{(n)}(j)$'s  are the coefficients of the polynomial
\begin{equation}\label{eq:S_i_is_vertex_cont}
(t+1)(t+3)\dots (t+2i-1)\times(t+1)(t+3)\dots (t+2(n-i)-1) = \sum_{j=0}^n P_{i}^{(n)}(j) t^j.
\end{equation}
In the one-dimensional case $d=1$, the convex hull is the interval
$$
C_n = \left[\min_{i=0,\dots,n}S_i, \max_{i=0,\dots,n} S_i\right],
$$
so (by symmetry of the increments) the probability that $S_i$ is a vertex of $C_n$ is just twice the probability that $S_i$ is the maximum. Therefore, \eqref{eq:S_i_is_vertex} and~\eqref{eq:S_i_is_vertex_cont} yield
\begin{equation}\label{eq:arcsine_maximum}
\P[\max\{S_0,\dots,S_n\} = S_i] =
\frac 1 {2^{2n}} \binom{2i}{i} \binom{2n-2i}{n-i}
= \frac {(2i-1)!!(2n-2i-1)!!}{(2i)!!(2n-2i)!!},
\end{equation}
for all $i=0,\dots,n$, which recovers the discrete arcsine law for the position of the maximum due to Sparre Andersen~\cite[Theorem~C]{sparre_andersen2}, see also~\cite[Vol II, Section XII.8]{Feller}.
Thus, we can view Theorem~\ref{1249} as a multidimensional generalization of the discrete arcsine law.
\end{remark}

\begin{remark}
For faces of maximal possible dimension $k = d-1$, Theorem~\ref{1249} (with only non-zero term $P^{(n)}_{i_1,\dots,i_d}(0)$ in the numerator) recovers a formula of Vysotsky and Zaporozhets~\cite{vysotsky_zaporozhets}:
$$
\P[\conv(S_{i_1},\dots,S_{i_{d}}) \in \mathcal F_{d-1}(C_n)] = 2 \frac{(2i_1-1)!!}{(2i_1)!!}\frac{(2n - 2i_d-1)!!}{(2n - 2i_d)!!} \prod_{j=1}^{d-1} \frac{1}{i_{j+1} - i_j}.
$$
\end{remark}
\begin{remark}
The coefficients $P_{i_1,\dots,i_{k+1}}^{(n)}(j)$ admit the following probabilistic interpretation. Consider the random variables
$$
K_n := \ind_{A_1} + \ind_{A_2}+ \dots + \ind_{A_n}, \quad L_n := \ind_{A_2} + \ind_{A_4} + \dots + \ind_{A_{2n}}, \quad n\in\N,
$$
where $A_1,A_2,\dots$ are independent events with $\P[A_m] = 1/m$, $m\in\N$.  The generating functions of $K_n$ and $L_n$ are given by
$$
\E t^{K_n} = \frac{t(t+1)\dots (t+n-1)}{n!}, \quad
\E t^{L_n} = \frac{(t+1)(t+3)\dots (t+2n-1)}{2^n n!}.
$$

It is well known that the number of cycles  of a uniform random permutation on $n$ elements has the same distribution as $K_n$. This can be deduced from the connection between random uniform permutations and the Chinese restaurant process; see~\cite[Section~3.1]{pitman_book}. To give a similar interpretation of $L_n$, consider the group of \emph{signed} permutations of the set $\{1,\dots,n\}$. Any such permutation can be written in the form
$$
\Sigma=
\begin{pmatrix}
1 & 2 &  \dots  & n \\
\eps_1 \sigma(1) & \eps_2\sigma(2) &  \dots &  \eps_n \sigma(n)
\end{pmatrix},
$$
where $\sigma$ is a permutation on $\{1,\dots,n\}$ and $\eps_1,\dots,\eps_n\in \{-1,+1\}$. The permutation $\sigma$ can be decomposed into cycles. We call a cycle $w_1 \to w_2\to \dots \to w_r\to w_1$ of $\sigma$ an \emph{even cycle} of the signed permutation $\Sigma$ if $\eps_{w_1}\dots \eps_{w_r} = +1$, i.e.\ if making a full turn along the cycle does not change the sign.
Clearly, for a uniformly chosen random signed permutation, any cycle is even with probability $1/2$, independently of all other cycles. This implies that the number of even cycles has the same distribution as $L_n$. The symmetric group and the group of signed permutations, acting on $\R^n$ as the reflection groups $A_{n-1}$ and $B_n$, will play a major role in our proofs.

Let us now return to Theorem~\ref{1249}. Let $L_{i_1}^{(0)},K_{i_2-i_1}^{(1)}, \dots, K_{i_{k+1}-i_k}^{(k)}, L_{n-i_{k+1}}^{(k+1)}$ be independent random variables with the same distributions as $L_{i_1},K_{i_2-i_1}, \dots, K_{i_{k+1}-i_k}, L_{n-i_{k+1}}$, respectively. Then Theorem~\ref{1249} states that
$$
\P[\conv(S_{i_1},\dots,S_{i_{k+1}}) \in \mathcal F_k(C_n)] =
2 \sum_{l=0}^\infty \P[L_{i_1}^{(0)} + K_{i_2-i_1}^{(1)} +  \dots + K_{i_{k+1}-i_k}^{(k)} + L_{n-i_{k+1}}^{(k+1)} = d-2l-1].
$$
Thus, the faces probabilities are similar to the distribution functions of the total number of cycles in a set of independent random permutations of the appropriate types.
\end{remark}
\begin{example}\label{ex:non_symm_RW}
In a sharp contrast with Theorem~\ref{theo:expected_walk}, the assumption of \textit{symmetric} exchangeability is essential in Theorem~\ref{1249}.  To see this,  consider i.i.d.\ standard normal random vectors  $\eta_1,\dots,\eta_n$ in $\R^d$ and define
$$
\xi_1(t) := 1 + t\eta_1,\quad \dots, \quad \xi_n(t) := 1 + t\eta_n, \quad t>0.
$$
Clearly, the random vectors $\xi_1(t),\dots,\xi_n(t)$ satisfy assumptions $(\text{Ex})$ and $(\text{GP})$ for all $t>0$. On the other hand, the corresponding random walk $S_i(t):= \xi_1(t)+\dots+\xi_i(t)$, $1\leq i\leq n$, starting at $S_0(t):=0$ satisfies
$$
p(t) := \P[S_0(t) \text{ is a vertex of } \conv(S_0(t),\dots,S_n(t))] = \P[0\notin \conv(S_1(t),\dots,S_n(t))],
$$
which converges to $1$ as $t\to 0$ because $\ind_{\{0 \in \conv(S_1(t),\dots,S_n(t))\}} \to 0$ a.s.\ as $t\to 0$. It follows that $p(t)$ cannot be given by~\eqref{eq:S_i_is_vertex} for sufficiently small $t$. The reason is the lack of central symmetry of the distribution of increments.
\end{example}

\subsection{Face probabilities for random bridges}
Random bridges are essentially random walks required to return to the origin after $n$ steps.
Formally, let $\xi_1,\dots,\xi_n$ be (in general, dependent) random vectors in $\R^d$ with partial sums $S_i= \xi_1+\dots+\xi_i$, $1\leq i\leq n$, and $S_0=0$.
We impose the following assumptions on the increments $\xi_1,\dots,\xi_n$:
\begin{itemize}
\item[$(\text{Br})$] \textit{Bridge property:} $S_n=\xi_1+\dots+\xi_n = 0$ a.s.
\item[$(\text{Ex})$] \textit{Exchangeability:} For every permutation $\sigma$ of the set $\{1,\dots,n\}$, we have the distributional equality
$$
(\xi_{\sigma(1)},\dots,  \xi_{\sigma(n)}) \eqdistr (\xi_1,\dots,\xi_n).
$$
\item[$(\text{GP}')$] \textit{General position:}
For every $1\leq i_1 < \dots < i_d \leq n-1$, the probability that the vectors $S_{i_1}, \dots, S_{i_d}$ are linearly dependent, is $0$.
\end{itemize}
The bridge starts and terminates at the origin: $S_0=S_n =0$ a.s. Let us stress that, unlike in the case of random walks, we don't need any central symmetry  assumption on the increments. As above, we denote by $C_n= \conv (S_0,\dots,S_n)$ the convex hull of $S_0,\dots,S_n$ and by $\mathcal F_k(C_n)$ the set of its $k$-faces, where $0\leq k\leq d-1$.


\begin{theorem}\label{1249bridge}
Let $(S_i)_{i=0}^n$ be a random bridge in $\R^d$ whose increments $\xi_1,\dots,\xi_n$ satisfy the above assumptions $(\text{Br})$, $(\text{Ex})$, $(\text{GP}')$.
Fix some $0\leq k\leq d-1$ and let $0 \le i_1 < \dots < i_{k+1}  < n$ be any indices. Then,
$$
\P[\conv(S_{i_1},\dots,S_{i_{k+1}}) \in \mathcal F_k(C_n)]
=
\frac{
2(Q_{i_1,\dots,i_{k+1}}^{(n)}(d-k-1) + Q_{i_1,\dots,i_{k+1}}^{(n)}(d-k-3)+\dots)
}{(i_2-i_1)!\dots  (i_{k+1}-i_k)! (n-i_{k+1}+i_1)!},
$$
where the $Q_{i_1,\dots,i_{k+1}}^{(n)}(j)$'s  are the coefficients of the polynomial
\begin{equation*}
\prod_{l=1}^{k+1} ((t+1)(t+2)\dots (t+i_{l+1}-i_l-1)) = \sum_{j=0}^{n-k-1} Q_{i_1,\dots,i_{k+1}}^{(n)}(j) t^j,
\end{equation*}
and we put $i_{k+2} = n+i_1$. For $j<0$ and $j>n-k-1$ we use the convention $Q_{i_1,\dots,i_{k+1}}^{(n)}(j):=0$.
\end{theorem}
\begin{remark}
For faces of maximal dimension $k=d-1$ the above formula simplifies to
$$
\P[\conv(S_{i_1},\dots,S_{i_{d}}) \in \mathcal F_{d-1}(C_n)]= \frac 2 {(i_2-i_1) \dots (i_d - i_{d-1}) (n-i_d + i_1)},
$$
which recovers a result obtained in~\cite{vysotsky_zaporozhets}.
\end{remark}
\begin{remark}
At the other extreme case, taking $k=0$ in Theorem~\ref{1249bridge} yields the following formula for the probability that $S_i$, where $0\leq i <n$, is a vertex of the convex hull $C_n$:
\begin{equation}\label{eq:exp_vertices_bridge}
\P[S_{i} \in \mathcal F_{0}(C_n)] = \frac 2 {n!} \left(\stirling{n}{d} + \stirling{n}{d-2}+\dots\right).
\end{equation}
Note that the result does not depend on $i$ which becomes quite straightforward  if one notices the cyclic exchangeability: $(\xi_1,\dots,\xi_n)$ has the same distribution as $(\xi_{i+1},\dots,\xi_n,\xi_1,\dots,\xi_{i-1})$ for all $i=0,\dots,n-1$.  Since $\stirling{n}{1} = (n-1)!$,  in the one-dimensional case $d=1$ formula~\eqref{eq:exp_vertices_bridge} reduces to the classical result of Sparre Andersen~\cite[Corollary~2]{sparre_andersen1} stating that
$$
\P[\max\{S_0,\dots,S_{n}\} = S_i]  = \frac 1 n, \quad i=0, \dots, n-1.
$$
\end{remark}

\subsection{Shift averages of face probabilities for general random walks. Connection to random bridges}
Finally, let us again turn to random walks. As was argued in Example~\ref{ex:non_symm_RW}, face probabilities for \emph{non-symmetric} exchangeable random walks do not enjoy distribution freeness. On the other hand, the next theorem states that certain shift averages of face probabilities are distribution-free.
\begin{theorem}\label{theo:1139}
Let $(S_i)_{i=0}^n$ be a random walk in $\R^d$ whose increments $\xi_1,\dots,\xi_n$ satisfy conditions $(\text{Ex})$ and $(\text{GP})$ but do not need to satisfy $(\pm\text{Ex})$.
Then, for all $0\leq k \leq d-1$ and for all indices $1\leq l_1 < \dots < l_k \leq n$,
\begin{multline*}
\frac 1 {n+1-l_k} \sum_{i=0}^{n-l_k} \P[\conv(S_{i}, S_{i+l_1},\dots, S_{i+l_{k}})\in \cF_k(C_n)]
\\=
\frac 1 {n+1} \sum_{i=0}^{n} \P[\conv(S_{i}, S_{i+l_1},\dots, S_{i+l_{k}})\in \cF_k(C_n)]
\\=
\frac{
2(Q_{0, l_1,\dots,l_{k}}^{(n+1)}(d-k-1) + Q_{0,l_1,\dots,l_{k}}^{(n+1)}(d-k-3)+\dots)
}{l_1! (l_2-l_1)!\dots  (l_{k}-l_{k-1})! (n+1-l_{k})!},
\end{multline*}
where in the second line we put $S_{i+l_j} = S_{(i+l_j)-(n+1)}$ if $i+l_j\geq n+1$.
\end{theorem}

According to Theorem~\ref{1249bridge}, the last expression is exactly the face probability of a random bridge of length $n+1$. This is due to a direct relation between convex hulls of random walks and random bridges, which is the essence of our proofs of Theorem~\ref{theo:expected_walk} and Theorem~\ref{theo:1139} (given below in Sections~\ref{sec:proof_numb_faces} and~\ref{sec:proof_shifted}, respectively). The main idea, explored in Section~\ref{subsec:walks_equals_bridge}, is to construct a random bridge from a non-symmetric random walk $S_0,\dots,S_n$ by adding the extra increment $\xi_{n+1} = -S_n$ and reshuffling the total $n+1$ increments randomly to enforce the exchangeability. For faces of maximal dimension $k=d-1$, a different proof of Theorem~\ref{theo:1139} (without the middle term)  was given by Vysotsky and Zaporozhets~\cite{vysotsky_zaporozhets}.

\section{Absorption probability for the joint convex hull}

\subsection{Connection to absorption probabilities}
Let us describe the idea of our proofs of Theorems~\ref{1249} an~\ref{1249bridge}. For concreteness, consider a symmetric random walk $(S_i)_{i=0}^n$ in the three-dimensional space $\R^3$. Given some $0\leq i_1 < i_2 \leq n$, we consider the probability that the segment $[S_{i_1},S_{i_2}]$ is an edge of the polytope $C_n = \conv(S_0,\dots,S_n)$. Denote by $l$ the line passing through the points $S_{i_1}$ and $S_{i_2}$, and let $h$ be any two-dimensional plane orthogonal to the line $l$. The intersection point of $l$ and $h$ is denoted by $P_0$.

Consider the orthogonal projection of the random walk $S_0,\dots,S_n$ on the plane $h$. Since the projection of the points $S_{i_1}$ and $S_{i_2}$ is $P_0$ (which we from now on view as the ``origin'' of the plane $h$),  we can split the projected random walk path into three components: the ``walk'' from the projection of $S_0$ to $P_0$ (which shall be time-reversed and sign-changed to be ``starting'' at $P_0$), the ``bridge'' from $P_0$ to $P_0$, and the ``walk'' from $P_0$ to the projection of $S_n$. The basic geometric observation underlying our proof of Theorems~\ref{1249} and \ref{1249bridge} is as follows: $[S_{i_1},S_{i_2}]$ is an edge of the convex hull $C_n$ if and only if
the point $P_0$ is a vertex of the joint convex hull of these three projected paths.
Thus, we need to compute the so-called \textit{non-absorption probability}, that is the probability that the interior of the joint convex hull of several random walks and bridges starting at the origin does not contain the origin. Problems of this type for just one random walk or random bridge were considered in our recent work~\cite{KVZ15}.

\subsection{Absorption probability for joint convex hulls}
Consider a collection of $s$ random walks and $r$ random bridges in $\R^d$ whose increments have a joint distribution invariant under the following transformations: we are allowed to perform any signed permutation of the increments inside any random walk, and any permutation of increments inside any random bridge. The next theorem provides a distribution-free formula for the probability that the joint convex hull of such random walks and bridges absorbs the origin. A particular case of this theorem was stated without proof in~\cite[Theorem~2.7]{KVZ15}. Before stating the theorem, we introduce necessary notation and assumptions.

Denote by $\Sym(n)$the symmetric group on a set of $n$ elements. Fix $s,r\in \N_0:= \N \cup \{0\}$ that do not vanish simultaneously, $n_1,\dots,n_s \in \N$, $m_1,\dots,m_r\in\N\bsl\{1\}$,  and consider $d$-dimensional random vectors
\begin{equation}\label{1942}
\xi_1^{(1)}, \dots, \xi_{n_1}^{(1)}, \;\; \dots,\;\; \xi_1^{(s)}, \dots, \xi_{n_s}^{(s)},\;\;
\eta_1^{(1)}, \dots, \eta_{m_1}^{(1)},\;\; \dots,\;\; \eta_1^{(r)}, \dots, \eta_{m_r}^{(r)}
\end{equation}
such that $\eta_1^{(j)}+ \dots+ \eta_{m_j}^{(j)}=0$ a.s.\ for every $1\leq j \leq r$.
Assume that for all permutations $\sigma^{(1)}\in \Sym(n_1)$, $\dots$, $\sigma^{(s)} \in \Sym(n_s)$, $\theta^{(1)}\in \Sym(m_1)$, $\dots$, $\theta^{(r)} \in \Sym(m_r)$ and all signs $\eps_1^{(1)}, \dots, \eps_{n_1}^{(1)}$, $\dots$, $\eps_1^{(s)}, \dots, \eps_{n_s}^{(s)} \in \{-1, +1\}$, we have the distributional equality
\begin{multline} \label{eq:invar_product}
\left(
\xi_1^{(1)}, \dots, \xi_{n_1}^{(1)}, \;\; \dots,\;\; \xi_1^{(s)}, \dots, \xi_{n_s}^{(s)},
\;\;
\eta_1^{(1)}, \dots, \eta_{m_1}^{(1)},\;\; \dots, \;\; \eta_1^{(r)}, \dots, \eta_{m_r}^{(r)}
\right)\\
\eqdistr
\left(
\eps_1^{(1)} \xi_{\sigma_1(1)}^{(1)}, \dots, \eps_{n_1}^{(1)} \xi_{\sigma_1(n_1)}^{(1)},
\;\;\dots, \;\;
\eps_{1}^{(s)}\xi_{\sigma_s(1)}^{(s)}, \dots, \eps_{n_s}^{(s)}\xi_{\sigma_s(n_s)}^{(s)},\right.
\\
\left.\eta_{\theta_1(1)}^{(1)}, \dots,  \eta_{\theta_1(m_1)}^{(1)},
\;\;\dots,\;\;
\eta_{\theta_r(1)}^{(r)}, \dots, \eta_{\theta_r(m_r)}^{(r)}
\right).
\end{multline}
Consider the collection of $s$ random walks $(S_l^{(1)})_{l=1}^{n_1},\dots, (S_l^{(s)})_{l=1}^{n_s}$ and $r$ random bridges $(R_l^{(1)})_{l=1}^{m_1},\dots, (R_l^{(r)})_{l=1}^{m_r}$ defined by
\begin{align*}
S_l^{(i)} &= \xi_1^{(i)} + \dots + \xi_{l}^{(i)}, \,  1\leq i \leq s,\,  1\leq l \leq n_i,\\
R_l^{(j)} &= \eta_1^{(j)} + \dots + \eta_{l}^{(j)}, \,  1\leq j \leq r,\,  1\leq l \leq m_j.
\end{align*}
Write $H$ for the joint convex hull of these walks and bridges, that is
\begin{equation}\label{eq:def_joint_convex_hull_H}
H=\conv \left(S_1^{(1)}, \dots, S_{n_1}^{(1)}, \;\; \dots,\;\; S_1^{(s)}, \dots, S_{n_s}^{(s)},
\;\;
R_1^{(1)}, \dots, R_{m_1-1}^{(1)}, \;\; \dots,\;\; R_1^{(r)}, \dots, R_{m_r-1}^{(r)}\right).
\end{equation}
\begin{theorem}\label{1435}
Assume that \eqref{eq:invar_product} holds and that any $d$ random vectors from the list on the right-hand side of~\eqref{eq:def_joint_convex_hull_H} are linearly independent with probability $1$. Then
\begin{equation}\label{eq:absorption}
\P[0\in H]=
\frac{2(P(d+1) + P(d+3)+\dots)}{2^{n_1} n_1!\dots 2^{n_s} n_s! m_1!\dots  m_r!},
\end{equation}
where the $P(j)$'s (which also depend on $s,r,n_1,\dots,n_s,m_1,\dots,m_r$) are the coefficients of the polynomial
\begin{equation}\label{eq:def_p_k}
\prod_{i=1}^s ((t+1)(t+3)\dots (t+2n_i-1))\times \prod_{l=1}^r ((t+1)(t+2)\dots (t+m_l-1)) = \sum_{j=0}^\infty P(j) t^j. 
\end{equation}
\end{theorem}

\begin{remark}
Theorem~\ref{1435} computes the so-called \emph{absorption probability}. The non-absorption probability is given by
\begin{equation}\label{eq:non_absorption}
\P[0\notin H]=
\frac{2(P(d-1) + P(d-3)+\dots)}{2^{n_1} n_1!\dots 2^{n_s} n_s! m_1!\dots  m_r!},
\end{equation}
with the usual convention $P(j):=0$ for $j < 0$. To see the equivalence of~\eqref{eq:absorption} and~\eqref{eq:non_absorption} note that
$$
\sum_{j=0}^{\infty} P(j) = 2^{n_1} n_1!\dots 2^{n_s} n_s! m_1!\dots  m_r!,
\quad
\sum_{j=0}^{\infty} (-1)^j P(j) = 0,
$$
obtained by taking $t=+1$ and $t= -1$ in~\eqref{eq:def_p_k}.
\end{remark}

\begin{remark}\label{rem:zero included}
We can include the joint starting point $0$ to the joint convex hull and consider $H_0:= \conv(H, 0)$. Then, under the general position assumption of Theorem~\ref{1435}, we have
$$\P[0 \notin H] = \P[0 \in \mathcal F_0(H_0)].$$
\end{remark}

\begin{remark}\label{rem:non_gen_position_absorption}
Without the general position condition, it holds that
\begin{equation*}
\P[0\in \Int H]\leq
\frac{2(P(d+1) + P(d+3)+\dots)}{2^{n_1} n_1!\dots 2^{n_s} n_s! m_1!\dots  m_r!}
\leq
\P[0\in H],
\end{equation*}
where $\Int H$ is the interior of $H$. We omit the proof of these inequalities because it is analogous to the proof of Proposition~2.10 in~\cite{KVZ15}.
\end{remark}

\section{Proof of Theorem~\ref{1435}}

\subsection{Symmetry groups and Weyl chambers}
In our recent work~\cite{KVZ15} we showed that absorption probabilities for random walks with symmetrically distributed increments (respectively, random bridges) can be interpreted geometrically using Weyl chambers of type $B_n$ (respectively, $A_{n-1}$).  We shall extend these ideas by  showing that Theorem~\ref{1435} concerning the convex hull of \emph{several} walks and bridges can be interpreted in terms of Weyl chambers corresponding to the \emph{direct product} of several reflection groups. The possibility of extension to direct products was mentioned without proof in Theorem~2.7 of~\cite{KVZ15}.
We start by recalling some relevant definitions.

The \emph{reflection group of type $B_n$} is the symmetry group of the regular cube $[-1,1]^n$ (or of its dual, the regular crosspolytope). The elements of this group act on $\R^n$ by permuting the coordinates in arbitrary way and multiplying any number of coordinates by $-1$. The number of elements of this group is $2^n n!$. We shall not distinguish between an abstract group and its action because this convenient for our purposes.

The \emph{reflection group of type $A_{n-1}$} is the symmetric group $\Sym(n)$ which acts on $\R^{n}$ by permuting the coordinates. The number of elements of this group is $n!$. The action of this group leaves the following hyperplane invariant:
$$
L_n= \{(x_1,\dots,x_{n})\in \R^{n}\colon x_1+\dots+x_{n} = 0\},
$$
which explains why the subscript $n-1$ rather than $n$ appears in the standard notation $A_{n-1}$.
Note that the group $A_{n-1}$ is the symmetry group of the regular simplex with $n$ vertices (defined as the convex hull of the standard basis in $\R^{n}$).

The {\it fundamental Weyl chambers} of type $A_{n-1}$ and $B_n$ are the following convex cones in $\R^n$:
\begin{align*}
\cC(A_{n-1}) &:=\{(x_1,\dots,x_{n})\in \R^n \colon x_1<x_2<\dots < x_{n}\},\\
\cC(B_n) &:= \{(x_1,\dots,x_n)\in\R^n\colon 0< x_1<x_2<\dots < x_n\}.
\end{align*}
Observe that $\cC(A_{n-1})$ is a {\it fundamental domain} for the reflection group $A_{n-1}$. This means that the cones of the form $g\cC(A_{n-1})$, $g \in A_{n-1}$, are pairwise disjoint and the union of their closures constitutes $\R^n$. The cones $g\cC(A_{n-1})$ or their closures will be referred to as \emph{Weyl chambers of type $A_{n-1}$}. Similarly, the cone $\cC(B_n)$ is a fundamental domain for the reflection group $B_n$, and the closures of the cones $g\cC(B_n)$, $g\in B_n$, are called \emph{Weyl chambers of type $B_n$}. Note that there are $n!$ Weyl chambers of type $A_{n-1}$ and $2^n n!$ Weyl chambers of type $B_n$.

In the sequel, a fundamental role will be played by the following \emph{reflection group of  direct product type}:
$$
G := B_{n_1} \times \dots \times B_{n_s}\times A_{m_1-1} \times \dots \times A_{m_r-1}.
$$
This group acts on $\R^{n_1}\times \dots \times \R^{n_s}\times \R^{m_1}\times \dots \times \R^{m_r} \equiv \R^n$,  where
$$
n=n_1+\dots+n_s+m_1+\dots+m_r,
$$
in the following natural way. Let  $e_1^{(i)},\dots, e_{n_i}^{(i)}$ be the standard basis of $\R^{n_i}$ (for all $1\leq i\leq s$) and let $f_1^{(j)},\dots, f_{m_j}^{(j)}$ be the standard basis of $\R^{m_j}$ (for all $1\leq j\leq r$). Then the elements of $G$ can be represented as tuples of the form
\begin{equation}\label{eq:g_def}
g=(g_{\sigma^{(1)}, \eps^{(1)}}, \dots, g_{\sigma^{(s)}, \eps^{(s)}},h_{\theta^{(1)}}, \dots, h_{\theta^{(r)}}),
\end{equation}
where: $\sigma^{(i)}\in \Sym(n_i),\theta^{(j)}\in \Sym(m_j)$ are permutations; $\eps^{(i)}:=(\eps_1^{(i)},\dots,\eps_{n_i}^{(i)})\in \{-1,+1\}^{n_i}$ are signs;  each $g_{\sigma^{(i)}, \eps^{(i)}}$ is the orthogonal transformation of $\R^{n_i}$ defined by
\begin{equation}\label{eq:g_def1}
g_{\sigma^{(i)},\eps^{(i)}} (e_k^{(i)}) = \eps_k^{(i)} e_{\sigma^{(i)}(k)}^{(i)}, \quad k=1,\dots,n_i;
\end{equation}
and each $h_{\theta^{(j)}}$ is the orthogonal transformation of $\R^{m_j}$ defined by
\begin{equation}\label{eq:g_def2}
h_{\theta^{(j)}} (f_l^{(j)}) = f_{\theta^{(j)}(l)}^{(j)}, \quad l=1,\dots,m_j.
\end{equation}
The total number of elements in the group  $G$ is $2^{n_1} n_1!\dots 2^{n_s} n_s!m_1!\dots  m_r!$.

\subsection{Absorption probability and subspaces intersecting Weyl chambers}
Consider the (open) Weyl chambers
\begin{align*}
C_B^{(i)} &:= \cC(B_{n_i}) = \{(x_1^{(i)}, \dots, x_{n_i}^{(i)})\in \R^{n_i} \colon 0< x_1^{(i)}< \dots < x_{n_i}^{(i)}\} \subset \R^{n_i},\\
C_A^{(j)} &:=\cC(A_{m_j-1}) = \{(y_1^{(j)}, \dots, y_{m_j}^{(j)})\in \R^{m_j} \colon y_1^{(j)}< \dots < y_{m_j}^{(j)}\} \subset \R^{m_j}
\end{align*}
and their direct product
$$
C:= C_B^{(1)}\times \dots \times C_B^{(s)}\times C_A^{(1)}\times \dots \times C_A^{(r)} \subset \R^{n_1}\times \dots \times \R^{n_s}\times \R^{m_1}\times \dots \times \R^{m_r}\equiv \R^n.
$$
Let $\bar C$ denote the closure of $C$.
Note that $C$ is a fundamental domain for the action of $G$ on $\R^n$. The closed convex cones $g \bar C$, where $g\in G$, are called \emph{Weyl chambers (of product type)}.
Let $L^{(j)}$ be the hyperplane invariant under the action of the group $A_{m_j-1}$:
\begin{equation}\label{eq:L_m_j_def}
L^{(j)}= \{(y_1,\dots,y_{m_j})\in \R^{m_j}\colon y_1+\dots+y_{m_j} = 0\}, \quad 1\leq j\leq r,
\end{equation}
and consider the linear subspace
$$
L:= \R^{n_1}\times \dots \times \R^{n_s}\times L^{(1)}\times \dots \times L^{(r)} \subset \R^n.
$$
Note that the action of $G$ leaves $L$ invariant. Let $A$ be a $d\times n$-matrix with the columns
\begin{equation}\label{eq:list_of_vectors}
\xi_1^{(1)},\dots,\xi_{n_1}^{(1)},\;\; \dots,\;\; \xi_1^{(s)},\dots,\xi_{n_s}^{(s)},\;\;
\eta_1^{(1)}, \dots, \eta_{m_1}^{(1)},\;\; \dots,\;\; \eta_1^{(r)}, \dots, \eta_{m_r}^{(r)}.
\end{equation}
We can view $A: \R^n \to \R^d$ as a linear operator mapping the standard basis of $\R^n$, namely
\begin{equation}\label{eq:standard_basis}
e_1^{(1)},\dots,e_{n_1}^{(1)},\;\;\dots,\;\; e_1^{(s)},\dots,e_{n_s}^{(s)},\;\;
f_1^{(1)}, \dots, f_{m_1}^{(1)},\;\; \dots,\;\; f_1^{(r)}, \dots, f_{m_r}^{(r)},
\end{equation}
to the vectors listed in~\eqref{eq:list_of_vectors}, respectively. The next lemma states that the absorption probability equals  the probability that the random linear subspace $(\Ker A)\cap L$ intersects any given Weyl chamber $g \bar C$ in a non-trivial way.
\begin{lemma}\label{1436}
Under the assumptions of Theorem~\ref{1435}, for every $g \in G$,
\begin{multline*}
\P[0\in \conv (S_1^{(1)}, \dots, S_{n_1}^{(1)},\;\; \dots,\;\; S_1^{(s)}, \dots, S_{n_s}^{(s)},\;\;
R_1^{(1)}, \dots, R_{m_1-1}^{(1)},\;\;\dots,\;\; R_1^{(r)}, \dots, R_{m_r-1}^{(r)})]
\\=\P[(\Ker A) \cap  L\cap (g\bar C) \neq \{0\}].
\end{multline*}
\end{lemma}
\begin{proof}
We are interested in the probability of the event
$$
E := \{(\Ker A) \cap L \cap  (g\bar C) \neq \{0\}\}= \{ \Ker (Ag) \cap L \cap \bar C \neq \{0\}\}.
$$
Recall that $g:\R^n\to\R^n$ is a linear operator given by~\eqref{eq:g_def}, \eqref{eq:g_def1}, \eqref{eq:g_def2}. The columns of the matrix $Ag$ are
\begin{align*}
&\eps_1^{(1)} \xi_{\sigma^{(1)}(1)}^{(1)},  \dots, \eps_{n_1}^{(1)} \xi_{\sigma^{(1)}(n_1)}^{(1)},
\;\;\dots, \;\;
\eps_{1}^{(s)}\xi_{\sigma^{(s)}(1)}^{(s)}, \dots, \eps_{n_s}^{(s)}\xi_{\sigma^{(s)}(n_s)}^{(s)},\\
& \eta_{\theta^{(1)}(1)}^{(1)}, \dots, \eta_{\theta^{(1)}(n_1)}^{(1)},
\;\;\dots,\;\; \eta_{\theta^{(r)}(1)}^{(r)}, \dots, \eta_{\theta^{(r)}(m_r)}^{(r)},
\end{align*}
as one can easily check by computing the action of $A g$ on the standard basis of $\R^n$; see~\eqref{eq:standard_basis}. So, we can write the event $E$ in the form
\begin{multline}\label{eq:E_product}
E = \Big\{\exists (x^{(1)},\dots, x^{(s)},y^{(1)},\dots, y^{(r)}) \in (\bar C_B^{(1)}\times \dots \times \bar C_B^{(s)}\times \bar C_A^{(1)}\times \dots \times \bar C_A^{(r)})\cap (L \bsl\{0\})
\colon\\
\sum_{i=1}^{s} \left(\eps_1^{(i)} \xi^{(i)}_{\sigma^{(i)}(1)} x_1^{(i)} + \dots + \eps_{n_i}^{(i)} \xi^{(i)}_{\sigma^{(i)}(n_i)} x_{n_i}^{(i)}\right)
+\sum_{j=1}^{r} \left( \eta^{(j)}_{\theta^{(j)}(1)} y_1^{(j)} + \dots + \eta^{(j)}_{\theta^{(j)}(m_j)} y_{m_j}^{(j)}\right)= 0\Big\}.
\end{multline}

For every $1 \le i \le s$, there is a bijective correspondence between $x^{(i)}= (x_1^{(i)},\dots,x_{n_i}^{(i)})\in \bar C_B^{(i)}$ and $\tilde{x}^{(i)} = (\tilde{x}_1^{(i)},\dots,\tilde{x}_{n_i}^{(i)})\in \R_{\geq 0}^{n_i}$ given by
$$
x_1^{(i)}=\tilde{x}_1^{(i)}, \;\; x_2^{(i)}=\tilde{x}_1^{(i)}+\tilde{x}_2^{(i)},\;\; \dots, \;\;  x_{n_i}^{(i)}= \tilde{x}_{1}^{(i)}+\dots+\tilde{x}_{n_i}^{(i)}.
$$
Similarly, there is a bijective correspondence between $y^{(j)}= (y_1^{(i)},\dots,y_{m_j}^{(j)})\in \bar C_A^{(j)}\cap L^{(j)}$ and $\tilde{y}^{(j)} = (\tilde{y}_1^{(j)},\dots,\tilde{y}_{m_j-1}^{(j)})\in \R_{\geq 0}^{m_j-1}$ given by
$$
\tilde{y}_1^{(j)}=y_2^{(j)}-y_1^{(j)},\;\; \dots, \;\; \tilde{y}_{m_j-1}^{(j)}=y_{m_j}^{(j)}-y_{m_j-1}^{(j)},
$$
or, equivalently,
$$
y_1^{(j)}=\tilde{y}_0^{(j)}, \;\; y_2^{(j)}=\tilde{y}_0^{(j)}+\tilde{y}_1^{(j)},\;\; \dots, \;\;
y_{m_j}^{(j)}= \tilde{y}_{0}^{(j)} + \tilde{y}_{1}^{(j)}+\dots+\tilde{y}_{m_j-1}^{(j)},
$$
where $\tilde{y}_0^{(j)}\in\R$ is chosen such that the condition $y_1^{(j)}+\dots+y_{m_j}^{(j)}=0$ holds.
Thus, we have
\begin{multline*}
E = \Big\{\exists (\tilde{x}^{(1)},\dots, \tilde{x}^{(s)},\tilde{y}^{(1)},\dots, \tilde{y}^{(r)}) \in
(\R_{\geq 0}^{n_1}\times \dots \times \R_{\geq 0}^{n_s}\times \R_{\geq 0}^{m_1-1}\times \dots \times \R_{\geq 0}^{m_r-1})\bsl \{0\}
\colon\\
\sum_{i=1}^{s} \sum_{k=1}^{n_i} \tilde{x}_k^{(i)} \left(\eps_k^{(i)} \xi^{(i)}_{\sigma^{(i)}(k)} + \dots + \eps_{n_i}^{(i)} \xi^{(i)}_{\sigma^{(i)}(n_i)}\right)
+\sum_{j=1}^{r} \sum_{l=1}^{m_j-1} \tilde{y}_l^{(j)} \left( \eta^{(j)}_{\theta^{(j)}(l+1)} + \dots +  \eta^{(j)}_{\theta^{(j)}(m_j)}\right)
 = 0\Big\}
\end{multline*}
modulo null sets, where we omitted the terms
$$
\tilde{y}_0^{(j)} \left(\eta^{(j)}_{\theta^{(j)}(1)} + \dots +  \eta^{(j)}_{\theta^{(j)}(m_j)}\right)
=
\tilde{y}_0^{(j)} \left(\eta_1^{(j)}+ \dots+ \eta_{m_j}^{(j)}\right) =0
\quad
\text{a.s.}, \quad 1 \le j \le r,
$$
which vanish by the bridge condition \eqref{1942}.
The invariance assumption~\eqref{eq:invar_product} implies the distributional equality
\begin{multline}
\left(\Big\{\eps_k^{(i)} \xi^{(i)}_{\sigma^{(i)}(k)} + \dots + \eps_{n_i}^{(i)}
\xi^{(i)}_{\sigma^{(i)}(n_i)}\Big\}_{\substack{i=1,\dots,s\\k=1,\dots, n_i}},
\Big\{\eta^{(j)}_{\theta^{(j)}(l+1)} + \dots +  \eta^{(j)}_{\theta^{(j)}(m_j)}\Big\}_{\substack{j=1,\dots,r\\l=1,\dots, m_j-1}}
\right)
\\
\eqdistr
\left(
\Big\{S_{n_i-k+1}^{(i)}\Big\}_{\substack{i=1,\dots,s\\k=1,\dots, n_i}},
\Big\{R_{m_j-l}^{(j)}\Big\}_{\substack{j=1,\dots,r\\l=1,\dots, m_j-1}}
\right).
\end{multline}
Therefore,
\begin{multline*}
\P[E]
= \P\Big[\exists (\tilde{x}^{(1)},\dots, \tilde{x}^{(s)},\tilde{y}^{(1)},\dots, \tilde{y}^{(r)}) \in
(\R_{\geq 0}^{n_1}\times \dots \times \R_{\geq 0}^{n_s}\times \R_{\geq 0}^{m_1-1}\times \dots \times \R_{\geq 0}^{m_r-1})\bsl \{0\}
\colon \\
\sum_{i=1}^s \left(\tilde{x}_1^{(i)} S_{n_i}^{(i)} + \tilde{x}_2^{(i)} S_{n_i-1}^{(i)} + \dots + \tilde{x}_{n_i}^{(i)} S_1^{(i)}\right)
\\+\sum_{j=1}^r \left(\tilde{y}_1^{(j)} R_{m_j-1}^{(j)} + \tilde{y}_2^{(j)} R_{m_j-2}^{(j)} + \dots + \tilde{y}_{m_j-1}^{(j)} R_1^{(j)}\right) = 0\Big].
\end{multline*}
The term on the right-hand side is the probability that the joint convex hull of the walks $S_k^{(i)}$, $1\leq k\leq n_i$, $1\leq i\leq s$, and the bridges $R_l^{(j)}$, $1\leq l\leq m_j-1$, $1\leq j\leq r$, contains $0$. This proves the lemma.
\end{proof}

\subsection{Hyperplane arrangements}
Now we need some results from the theory of hyperplane arrangements~\cite{OT92,rS07}.
A \emph{linear hyperplane arrangement} (or simply ``\textit{arrangement}'') $\cA$ is a finite set of distinct hyperplanes in $\R^n$ that pass through the origin.
The \emph{rank} of an arrangement $\cA$ is the codimension of the intersection of all hyperplanes in the arrangement:
$$
\rank(\cA)=n-\dim\left(\bigcap_{H\in\cA}H\right).
$$
Equivalently, the rank is  the dimension of the space spanned by the normals to the hyperplanes in $\cA$.
The \emph{characteristic polynomial} $\chi_{\cA}(t)$ of the arrangement $\cA$ is defined by
\begin{equation}\label{1459}
\chi_{\cA}(t)=\sum_{\cB\subset\cA}(-1)^{\#\cB}t^{n-\rank(\cB)},
\end{equation}
where $\#\cB$ denotes the number of elements in the set $\cB$, and $\rank(\varnothing) = 0$ under convention that the intersection over the empty set of hyperplanes is $\R^n$. The original definition of the characteristic polynomial uses the notions of the intersection poset of $\cA$ and the M\"obius function on it; see~\cite[Section~1.3]{rS07}. The equivalence of both definitions was proved by Whitney; see, e.g., \cite[Lemma~2.3.8]{OT92} or~\cite[Theorem~2.4]{rS07}.

Denote by $\cR(\cA)$  the finite set of open connected components (``\emph{regions}'' or ``\emph{chambers}'') of the complement $\R^n\setminus\cup_{H\in\cA} H$ of the hyperplanes. 
The following fundamental result due to Zaslavsky~\cite{tZ75} (see also~\cite[Theorem~2.5]{rS07}) expresses the number of regions of the arrangement $\cA$ in terms of its characteristic polynomial:
\begin{equation}\label{1112}
\# \cR(\cA)=(-1)^n\chi_{\cA}(-1).
\end{equation}


The \emph{lattice} $\cL(\cA)$ generated by an arrangement $\cA$ in $\R^n$ consists of all linear subspaces that can be represented as intersections of some of the hyperplanes from $\cA$, that is
$$
\cL(\cA) = \left\{\bigcap_{H\in\cB}H \colon \cB \subset \cA\right\}.
$$
By definition, $\R^n\in \cL(\cA)$, corresponding to the empty intersection over $\cB=\varnothing$. Let $M_{n-d}$ be a linear subspace in $\R^n$ of codimension $d\leq n-1$. We say that $M_{n-d}$ is in \emph{general position} with respect to $\cA$ if for all $K\in \cL(\cA)$,
\begin{equation}\label{1222}
\dim (M_{n-d}\cap K) =
\begin{cases}
\dim K -  d, &\text{if } \dim K \geq d,\\
0, &\text{if } \dim K \leq d.
\end{cases}
\end{equation}
The next theorem provides a formula for the number of regions in $\cR(\cA)$ intersected by a linear subspace in general position. We refer to~\cite[Theorem~3.3 and Lemma~3.5]{KVZ15} for its proof.
\begin{theorem}\label{1229}
Let $M_{n-d}$ be a linear subspace in $\R^n$ of codimension $d$ that is in general position w.r.t.\ to a linear hyperplane arrangement $\cA$. Let
\begin{equation}\label{eq:chi_def}
\chi_{\cA}(t)=\sum_{k=0}^n (-1)^{n-k} a_kt^k
\end{equation}
be the characteristic polynomial of $\cA$. Then, the number of regions in $\cR(\cA)$ intersected by $M_{n-d}$ is given by
\begin{align*}
\#\{R\in \cR(\cA)\colon R\cap M_{n-d}\ne\varnothing\}
&=
\#\{R\in \cR(\cA)\colon \overline R\cap M_{n-d}\ne \{0\}\}\\
&=
2(a_{d+1} + a_{d+3} +\dots)
,
\end{align*}
where we put $a_k=0$ for $k\notin\{0,\dots,n\}$.
\end{theorem}

Let us consider a special case: the \emph{reflection arrangements} in $\R^n$ of types $A_{n-1}$ and $B_n$.
These arrangements consist  of the hyperplanes
\begin{align}
\cA(A_{n-1})&\colon \quad  \{x_i = x_j\}, \quad 1\leq i < j \leq n, \label{eq:arr_A}\\
\cA(B_n)&\colon \quad \{x_i = x_j\}, \quad \{x_i = -x_j\}, \quad \{x_k = 0\}, \quad 1\leq i < j \leq n, \quad 1\leq k\leq n, \label{eq:arr_B}
\end{align}
where $(x_1,\dots,x_n)$ are the coordinates on $\R^n$.
It is easily seen that the regions in $\cR(\cA(A_{n-1}))$ and $\cR(\cA(B_n))$ are precisely the interiors of the Weyl chambers of type $A_{n-1}$ and $B_n$.

The characteristic polynomials of the reflection arrangements (see Section~5.1 and Corollary 2.2 in~\cite{rS07}) are given by
\begin{align}
&\chi_{\cA(A_{n-1})}(t) = t (t-1) \dots (t-(n-1))= \sum_{k=1}^{n} (-1)^{n-k} \stirling{n}{k} t^k,  \label{eq:chi_A} \\
\chi&_{\cA(B_{n})}(t) = (t-1)(t-3)\dots (t-(2n-1)) = \sum_{k=0}^{n} (-1)^{n-k} B(n,k)t^k, \label{eq:chi_B}
\end{align}
where $\stirling{n}{k}$ (the Stirling numbers of the first kind) and $B(n,k)$ (their $B$-analogues) have the following generating functions:
\begin{equation*}
t(t+1) \dots (t+n-1) = \sum_{k=1}^n \stirling{n}{k} t^k,
\quad
(t+1)(t+3)\dots (t+2n-1) = \sum_{k=0}^n B(n,k) t^k.
\end{equation*}

\subsection{Proof of Theorem~\ref{1435}}
We are now ready to complete the proof of Theorem~\ref{1435}. Applying Lemma~\ref{1436} to all $g\in G$ and taking the arithmetic mean, we obtain
\begin{align}\label{eq:reduction}
\P[0\in H]
=
\frac 1 {\# G} \sum_{g\in G} \P[(\Ker A) \cap L \cap (g\bar C) \neq \{0\}]
=
\frac {\E N} {\# G} ,
\end{align}
where the random variable
\begin{equation} \label{eq:N=}
N := \sum_{g\in G} \ind_{\{(\Ker A) \cap L\cap  (g\bar C) \neq \{0\}\}}
\end{equation}
counts the number of Weyl chambers of the form $g\bar C$, $g\in G$, intersected by the random linear subspace $(\Ker A)\cap L$ in a nontrivial way.

Given arbitrary arrangements $\cA_1,\dots,\cA_M$ in $\R^{q_1},\dots,\R^{q_M}$, define their \emph{direct product} as the following arrangement in $\R^q \equiv \R^{q_1+\dots+q_M}$:
\begin{multline*}
  \cA_1\times\dots\times\cA_M= \\
  \{ H \times \R^{q-q_1} \}_{H \in \cA_1}\;\bigcup\; \{ \R^{q_1} \times H \times \R^{q-q_1-q_2}\}_{H \in \cA_2} \;\bigcup\; \dots
  \;\bigcup\; \{ \R^{q-q_M} \times H\}_{H \in \cA_M}.
\end{multline*}
Consider the reflection arrangement $\cA$ of type $B_{n_1} \times \dots \times B_{n_s}\times A_{m_1-1} \times \dots \times A_{m_r-1}$, that is
$$
\cA=\cA(B_{n_1}) \times \dots \times \cA(B_{n_s})\times \cA(A_{m_1-1}) \times \dots \times \cA(A_{m_r-1}).
$$
The characteristic polynomial of a direct product of arrangements is the product of the individual characteristic polynomials (Lemma~2.50 on p.~43 in~\cite{OT92}), hence
\begin{align*}
(-1)^{n} \chi_{\cA} (-t)
&=\prod_{i=1}^s ((t+1)(t+3)\dots (t+2n_i-1)) \times\prod_{j=1}^r (t(t+1)\dots (t+m_j-1)) \\
&=
\sum_{k=r}^{n+r} P(k-r) t^k,
\end{align*}
where we used the notation $P(k)$ from~\eqref{eq:def_p_k}. Now observe that $N$ is the number of regions in $\cR(\cA)$ intersected by $(\Ker A) \cap L$.
\begin{lemma}\label{lem:gen_pos}
If the general position assumption imposed in Theorem~\ref{1435} holds, then with probability $1$, the random linear subspace $(\Ker A)\cap L$ has codimension $d+r$ in $\R^n$ and is in general position w.r.t.\ $\cA$.
\end{lemma}
Postponing the proof of the lemma for a moment, we apply Theorem~\ref{1229} to obtain that
$$
N = 2 (P(d+1) + P(d+3) +\dots)
\quad \text{a.s.}
$$
Combining this equation with \eqref{eq:reduction} completes the proof of Theorem~\ref{1435}. \hfill $\Box$

\begin{proof}[Proof of Lemma~\ref{lem:gen_pos}]
In the case of just one random walk or random bridge, we proved the lemma in~\cite[Section 6.2]{KVZ15}. The proof in the direct product case is similar and we sketch only the main ideas. Consider a linear subspace $K$ from the lattice generated by the arrangement $\cA$. That is, $K$ can be represented as an intersection of some hyperplanes from $\cA$ and, consequently,
\begin{equation}\label{eq:K_product}
K= K_1\times \dots \times K_{s} \times K_1'\times \dots \times K_r'
\end{equation}
for some $K_i\in \cL(\cA(B_{n_i}))$, $1\leq i\leq s$, and $K_j'\in \cL(\cA(A_{m_j-1}))$, $1\leq j\leq r$.
Our aim is to prove that
\begin{equation}\label{eq:gen_pos_need0}
\dim (K\cap L \cap \Ker A) \stackrel{\text{a.s.}}{=}
\begin{cases}
\dim K - d - r, & \text{ if } \dim K \geq d+r,\\
0, & \text{ if } \dim K \leq d+r.
\end{cases}
\end{equation}
Note in passing that taking $K=\R^n$ would yield $\codim (L \cap \Ker A)= d+r$ a.s.

In fact, it suffices to prove that
\begin{equation}\label{eq:gen_pos_need}
\dim (K\cap \Ker A) \stackrel{\text{a.s.}}{=}
\begin{cases}
\dim K - d, & \text{ if } \dim K \geq d+r,\\
r, & \text{ if } \dim K \leq d+r.
\end{cases}
\end{equation}
To see that~\eqref{eq:gen_pos_need} implies~\eqref{eq:gen_pos_need0}, let us show that $K\cap \Ker A$ contains the $r$-dimensional linear subspace $L^{\bot}$ with probability $1$. Indeed, for every $1\leq j\leq r$ we have
$$
A(f_1^{(j)} + \dots + f_{m_j}^{(j)}) = \eta_1^{(j)} + \dots + \eta_{m_j}^{(j)} = 0 \quad \text{a.s.}
$$
by definition of $A$ and the bridge property, whence $L^{\bot} \subset \Ker A$. To see that $L^{\bot} \subset K$, recall that by definition of the arrangement of type $A_{m_j-1}$, see~\eqref{eq:arr_A}, the vector $f_1^{(j)} + \dots + f_{m_j}^{(j)}$ belongs to all hyperplanes from $\cA(A_{m_j-1})$ and hence, to all linear subspaces from $\cL(\cA(A_{m_j-1}))$.

Next we are going to write down an explicit system of equations defining $K$. Recall that $(x_1^{(i)}, \dots, x_{n_i}^{(i)})$ are coordinates on $\R^{n_i}$, while $(y_1^{(j)}, \dots, y_{m_j}^{(j)})$ are coordinates on $\R^{m_j}$. Let us first look at the lattice generated by the hyperplane arrangement $\cA(A_{m_j-1})$; see~\eqref{eq:arr_A} for its definition.  Any linear subspace belonging to this lattice is given by a system of equations of the following type. Decompose the variables $y_{1}^{(j)},\dots, y_{m_j}^{(j)}$ into some number, say $q(j)$, of non-empty groups, and then require the variables inside the same group to be equal to each other. Linear subspaces belonging to  the lattice generated by the hyperplane arrangement $\cA(B_{n_i-1})$, see~\eqref{eq:arr_B} for its definition, can be described as follows. Decompose the variables $x_1^{(i)}, \dots, x_{n_i}^{(i)}$ into some number, say $p(i)+1$ of groups (all groups being non-empty except possibly the last one). Require the variables in the last group to be $0$. For each group except the last one,  multiply each variable in the group by either $+1$ or $-1$, and require the resulting signed variables to be equal to each other.

Taking all the equations described above together, we obtain a system of equations defining $K$.  However, since the distribution of the linear subspace $(\Ker A)\cap L$ is invariant w.r.t.\ the action of $G$, after transforming everything by a suitable $g\in G$, we can assume without loss of generality that $K$ is given by the following simplified system of equations.
For every $1\leq i \leq s$, we have the equations
\begin{align*}
&\gamma_1(i):= x^{(i)}_1 = \dots = x^{(i)}_{u_{1}(i)},\\
&\gamma_2(i) := x^{(i)}_{u_{1}(i) +1} = \dots = x^{(i)}_{u_{2}(i)},\\
&\dots,\\
&\gamma_{p(i)}(i):= x^{(i)}_{u_{p(i)-1}(i)+1} = \dots = x^{(i)}_{u_{p(i)}(i)},\\
&x^{(i)}_{u_{p(i)}(i)+1} = \dots = x^{(i)}_{n_i} = 0,
\end{align*}
with some $0=:u_{0}(i) < u_1(i) < \dots < u_{p(i)}(i) \leq n_i$, and for every $1\leq j\leq r$, we have the equations
\begin{align*}
&\delta_1(j):=  y^{(j)}_1 = \dots = y^{(j)}_{v_{1}(j)},\\
&\delta_2(j) := y^{(j)}_{v_{1}(j) +1} = \dots = y^{(j)}_{v_{2}(j)},\\
&\dots,\\
&\delta_{q(j)}(j):= y^{(j)}_{v_{q(j)-1}(j)+1} = \dots = y^{(j)}_{m_j},
\end{align*}
with some $0=:v_{0}(j) < v_1(j) < \dots < v_{q(j)}(j) := m_j$. We use the variables $\gamma_1(i), \dots, \gamma_{p(i)}(i)$ ($1\leq i \leq s$) and $\delta_1(j), \dots, \delta_{q(j)}(j)$ ($1\leq j \leq r$) as coordinates on $K$. Note that
\begin{equation}\label{eq:dim_K}
\dim K = \sum_{i=1}^s p(i) + \sum_{j=1}^r q(j).
\end{equation}
The linear subspace $\Ker A$ is the given by the equation
\begin{equation}\label{eq:Ker_A}
\sum_{i=1}^s \sum_{l=1}^{n_i} x_l^{(i)} \xi_l^{(i)}   + \sum_{j=1}^r \sum_{l=1}^{m_j} y_l^{(j)} \eta_l^{(j)} = 0.
\end{equation}
Inside $K$, the linear subspace $K\cap \Ker A$ is given by the equation
\begin{multline}\label{eq:gen_pos_main}
\sum_{i=1}^s
\left(\gamma_1(i) S^{(i)}_{u_1(i)} +   \gamma_2(i) (S^{(i)}_{u_2(i)} - S^{(i)}_{u_1(i)}) +\dots + \gamma_{p(i)}(i) (S^{(i)}_{u_{p(i)}(i)} - S^{(i)}_{u_{p(i)-1}(i)}) \right)
\\+
\sum_{j=1}^r \left( \delta_1(j) R^{(j)}_{v_1(j)} +   \delta_2(j) (R^{(j)}_{v_2(j)} - R^{(j)}_{v_1(j)}) +\dots + \delta_{q(j)}(j) (0 - R^{(j)}_{v_{q(j)-1}(j)})\right)
\stackrel{\text{a.s.}}{=}0.
\end{multline}
Recall that the random walks and bridges take values in $\R^d$, so that, effectively, \eqref{eq:Ker_A} and~\eqref{eq:gen_pos_main} are systems of $d$ equations each.

Let $\dim K \geq d+r$. Then, by the general position assumption from Theorem~\ref{1435}, the collection of random vectors
$$
S^{(i)}_{u_1(i)}, S^{(i)}_{u_2(i)},\dots, S^{(i)}_{u_{p(i)}(i)}, \;\;\; (1\leq i \leq s),\;\;
R^{(j)}_{v_1(j)}, R^{(j)}_{v_2(j)}, \dots, R^{(j)}_{v_{q(j)-1}(j)}
\;\; (1\leq j\leq r)
$$
spans linearly the whole $\R^d$ with probability $1$ since the total number of the vectors is at least $d$; see~\eqref{eq:dim_K}. It follows that the system of $d$ equations in~\eqref{eq:gen_pos_main} has full rank a.s., hence the dimension of the set of its solutions is $\dim K - d$ a.s., thus proving the first case of~\eqref{eq:gen_pos_need}.  Let now $\dim K \leq d+r$. Then, we can find a linear subspace $K'\supset K$ such that $\dim K'= d+r$ and $K'\in \cL(\cA)$. Applying the above to $K'$, we obtain $\dim (K'\cap L \cap \Ker A) = 0$ a.s., hence $\dim (K\cap L \cap \Ker A) = 0$ a.s., thus proving the second case in~\eqref{eq:gen_pos_need}.
\end{proof}

\section{Proof of Theorems~\ref{1249} and~\ref{1249bridge}}
\begin{proof}[Proof of Theorem~\ref{1249}]
Given $k+1$ vectors $x_1,\dots,x_{k+1}\in\R^d$  denote by $\aff(x_1,\dots,x_{k+1}) = x_1 + \lin(0,x_2 - x_1, \dots,x_{k+1} - x_1)$ their affine hull and by
$$
\aff^\perp(x_1,\dots,x_{k+1}) = (\aff(x_1,\dots,x_{k+1})- x_1)^\perp
$$
the orthogonal complement of $\aff(x_1,\dots,x_{k+1})$, which is a linear subspace.

Let $\cdot|M$ denote the orthogonal projection on $M:=\aff^\perp(S_{i_1},\dots,S_{i_{k+1}})$. Note that $\dim M = d - k$ a.s.\ because
$$
(S_{i_2} -S_{i_1},\dots, S_{i_{k+1}} - S_{i_1}) \eqdistr (S_{i_2-i_1},\dots, S_{i_{k+1}-i_1})
$$
by $(\pm\text{Ex})$ (in fact, condition $(\text{Ex})$ suffices) and the random vectors on the right-hand side are a.s.\ linearly independent by $(\text{GP})$. Projecting the path $S_0,\dots,S_n$ on $M$ gives a random walk terminating at
$P_0:= S_{i_1}|M = \dots= S_{i_{k+1}}|M$ (viewed as the origin of $M$), $k$ random bridges that  start and terminate at $P_0$, and a random walk starting at $P_0$. The first walk shall be time-reversed and sign-changed to start from~$P_0$. The increments of the random walks are given by
\begin{align*}
&\xi_1^{(1)}=-\xi_{i_1}|M, \;\;  \xi_2^{(1)}=-\xi_{i_1-1}|M, \;\;  \dots, \;\; \xi_{i_1}^{(1)}=-\xi_1|M,\\
&\xi_1^{(2)}=\xi_{i_{k+1}+1}|M,\;\; \xi_2^{(2)}=\xi_{i_{k+1}+2}|M, \;\; \dots, \;\; \xi_{n-i_{k+1}}^{(2)}=\xi_n|M,
\end{align*}
while the increments of the random bridges are given by
$$
\eta_1^{(j)}=\xi_{i_j+1}|M, \;\; \dots, \;\; \eta_{i_{j+1}-i_j}^{(j)}=\xi_{i_{j+1}}|M,
\quad j=1,\dots k.
$$

We shall apply Theorem~\ref{1435} to these $s=2$ random walks and $r=k$ random bridges in $M$ with $M \cong \R^{d-k}$ a.s.
It is easy to see that  their increments listed above satisfy the invariance assumption \eqref{eq:invar_product} of Theorem~\ref{1435}, namely,  permuting the increments within the walks/bridges and changing the signs of the increments in both random walks does not change the joint distribution of the increments.  In fact, such transformations of the unprojected increments of the original random walk $S_1, \ldots, S_n$ do not change the joint distribution of these increments and, importantly, do not change~$M$.

Denote by $H_0$ the joint convex hull of the above random walks and bridges: $H_0 := C_n|M$. The key observation is as follows:
\begin{equation} \label{eq: key}
\conv(S_{i_1},\dots,S_{i_{k+1}}) \text{ is a }k\text{-face of }C_n  \text{ if and only if } P_0 \text{ is a vertex of }H_0
\end{equation}
on the set of full probability described by the general position assumption $(\text{GP})$.
This is evident since, by definition, the faces of a convex polytope are obtained by intersecting the polytope with its supporting hyperplanes.

Postponing the verification of the general position assumption for a moment, we apply Theorem~\ref{1435}, see also Remark~\ref{rem:zero included} and~\eqref{eq:non_absorption},  to obtain that
$$
\P[P_0 \in \mathcal{F}_0(H_0)]=
\frac{2(P_{i_1,\dots,i_{k+1}}^{(n)}(d-k-1) + P_{i_1,\dots,i_{k+1}}^{(n)}(d-k-3)+\dots)}{2^{i_1+n-i_{k+1}} i_1!(i_2-i_1)!\dots  (i_{k+1}-i_k)! (n-i_{k+1})!}
$$
with the generating function for the $P_{i_1,\dots,i_{k+1}}^{(n)}(j)$'s defined in Theorem~\ref{1249}.

To complete the proof of  Theorem~\ref{1249} we need to verify the general position assumption of Theorem~\ref{1435}. Let $T_1,\dots,T_{d-k}$ be any $d-k$ random vectors from the list
\begin{align*}
&S_{i_1-1} - S_{i_1}, S_{i_1-2} - S_{i_1}, \dots, S_{1} - S_{i_1}    \quad \text{(first walk, unprojected and time-reversed)}, \\
&S_{i_1+1} - S_{i_1}, S_{i_1+2} - S_{i_1}, \dots,S_{i_2-1} - S_{i_1}    \quad \text{(first bridge, unprojected)}, \\
&\dots,\\
&S_{i_k+1} - S_{i_k}, S_{i_k+2} - S_{i_k}, \dots,S_{i_{k+1}-1} - S_{i_k}    \quad \text{($k$-th bridge, unprojected)}, \\
&S_{i_{k+1}+1} - S_{i_{k+1}}, S_{i_{k+1}+2} - S_{i_{k+1}}, \dots,S_{n} - S_{i_{k+1}} \quad \text{(second walk, unprojected)}.
\end{align*}
We need to show that  $T_1|M,\dots, T_{d-k}|M$ are linearly independent with probability $1$. Since the orthogonal complement of $h$ is spanned by $S_{i_2}-S_{i_1},\dots, S_{i_{k+1}} - S_{i_k}$, it suffices to check that the random vectors
\begin{equation}\label{eq:list_GP_face_probab}
T_1,\dots,T_{d-k}, S_{i_2}-S_{i_1},\dots, S_{i_{k+1}} - S_{i_k}
\end{equation}
are linearly independent with probability $1$. But it is easy to see that their linear hull coincides with the linear hull of
$$
S_{j_1}-S_{j_0}, S_{j_2}-S_{j_1}, \dots, S_{j_d}-S_{j_{d-1}}
$$
for some collection of indices $0\leq j_0< j_1 < \dots < j_d \leq n$ containing the set $\{i_1, \ldots, i_{k+1}\}$. By assumptions $(\pm\text{Ex})$ and $(\text{GP})$, this linear hull has maximal possible dimension $d$ a.s. This proves the a.s. linear independence of the random vectors in~\eqref{eq:list_GP_face_probab}.
\end{proof}

\begin{proof}[Proof of Theorem~\ref{1249bridge}]
The main idea is the same as in the previous proof. Consider the linear subspace $M:=\aff^\perp(S_{i_1},\dots,S_{i_{k+1}})$ and note that $\dim M = d-k$ a.s.\ by the same argument as in the previous proof. Projecting the closed path $S_0,\dots,S_n$ on $M$, we obtain $k+1$ random bridges in $M$ (with $M \cong \R^{d-k}$ a.s.) starting and terminating at $P_0:=S_{i_1}|M = \dots= S_{i_{k+1}}|M$. The random bridge number $j+1\in \{2,\dots,k+1\}$ is the projection of the path $S_{i_j}, S_{i_j+1},\dots, S_{i_{j+1}}$ and has increments
$$
\eta_1^{(j+1)}=\xi_{i_j+1}|M, \;\;  \eta_2^{(j+1)}=\xi_{i_j+2}|M, \;\;  \dots, \;\; \eta_{i_{j+1} - i_j}^{(j+1)}=\xi_{i_{j+1}}|M, \quad j=1,\dots,k,
$$
while the first random bridge is the projection of the path $S_{i_{k+1}},\dots, S_{n-1}, 0, S_1,\dots, S_{i_1}$ and its increments are
$$
\eta_1^{(1)}=\xi_{i_{k+1}+1} | M, \;\; \dots, \;\; \eta_{n-i_{k+1}}^{(1)}=\xi_{n}|M,\;\;
\eta_{n-i_{k+1} + 1}^{(1)}=\xi_{1} | M, \;\; \dots, \;\; \eta_{n-i_{k+1} + i_1}^{(1)}=\xi_{i_1}|M.
$$

Again, we observe that the invariance condition~\eqref{eq:invar_product} of Theorem~\ref{1435} is satisfied for these $r=k+1$ random bridges (and $s=0$ random walks) because the joint distribution of the increments is invariant with respect to arbitrary permutations of the increments within the bridges. The general position assumption of Theorem~\ref{1435} will be verified below. Observe that with probability one, $\conv(S_{i_1},\dots,S_{i_{k+1}})$ is a $k$-face of $C_n$ if and only if $P_0$ is a vertex of the joint convex hull $H_0:=C_n|M$ of the above bridges; see~\eqref{eq: key}. 
Hence, Theorem~\ref{1435} (see also Remark~\ref{rem:zero included} and~\eqref{eq:non_absorption}), yields
$$
\P[P_0 \in \mathcal{F}_0(H_0)]=
\frac{
2(Q_{i_1,\dots,i_{k+1}}^{(n)}(d-k-1) + Q_{i_1,\dots,i_{k+1}}^{(n)}(d-k-3)+\dots)
}{(i_2-i_1)!\dots  (i_{k+1}-i_k)! (n-i_{k+1}+i_1)!}
$$
with the generating function for the $Q_{i_1,\dots,i_{k+1}}^{(n)}(j)$'s defined in Theorem~\ref{1249bridge}.

To verify the general position assumption of Theorem~\ref{1435},  let $T_1,\dots,T_{d-k}$ be any $d-k$ vectors from the list
\begin{align*}
&S_{i_{k+1}+1} - S_{i_{k+1}}, S_{i_{k+1}+2} - S_{i_{k+1}}, \dots, S_{n-1}- S_{i_{k+1}},
 0 - S_{i_{k+1}}, S_1-S_{i_{k+1}},\dots,S_{i_1-1}-S_{i_{k+1}},  \\
&S_{i_1+1} - S_{i_1}, S_{i_1+2} - S_{i_1}, \dots,S_{i_2-1} - S_{i_1}, \\
&\dots,\\
&S_{i_k+1} - S_{i_k}, S_{i_k+2} - S_{i_k}, \dots,S_{i_{k+1}-1} - S_{i_k}.
\end{align*}
Our aim is to prove that $T_1|M,\dots, T_{d-k}|M$ are linearly independent with probability $1$. The orthogonal complement of $h$ is spanned by $S_{i_2}-S_{i_1},\dots, S_{i_{k+1}} - S_{i_k}$, hence our task reduces to showing that the vectors
\begin{equation}\label{eq:list_GP_face_probab_bridge}
T_1,\dots,T_{d-k}, S_{i_2}-S_{i_1},\dots, S_{i_{k+1}} - S_{i_k}
\end{equation}
are linearly independent with probability $1$. But their linear hull coincides with the linear hull of
$$
S_{j_1}-S_{j_0}, S_{j_2}-S_{j_1}, \dots, S_{j_d}-S_{j_{d-1}}
$$
for a suitable collection of indices $i_1 =: j_0< j_1 <\dots < j_d < n+j_0$ containing the set $\{i_1, \ldots, i_{k+1}\}$ (with the convention $S_{n+j} = S_j$ for $j \ge 0$). By assumptions $(\text{Ex})$ and $(\text{GP}')$, this linear hull has maximal possible dimension $d$ a.s. This proves the a.s. linear independence of the random vectors in~\eqref{eq:list_GP_face_probab_bridge}.
\end{proof}

\section{Expected number of faces of a random walk} \label{sec:proof_numb_faces}
\subsection{Method of proof}
In the following we shall sketch the main steps in the proof of Theorem~\ref{theo:expected_walk}.
As a direct corollary of Theorem~\ref{1249}, we obtain a formula for the expected number of $k$-dimensional faces of $C_n$ under assumptions $(\pm\text{Ex})$ and $(\text{GP})$:
\begin{equation}\label{eq:expect_faces_walk_big_sum}
\E [f_k(C_n)]
=
2\sum_{0\leq i_1<\dots<i_{k+1}\leq n}\frac{P_{i_1,\dots,i_{k+1}}^{(n)}(d-k-1) + P_{i_1,\dots,i_{k+1}}(d-k-3)+\dots}{2^{i_1+n-i_{k+1}} i_1!(i_2-i_1)!\dots  (i_{k+1}-i_k)! (n-i_{k+1})!},
\end{equation}
for all $0\leq k\leq d-1$.  Similarly, it follows from Theorem~\ref{1249bridge} that for random bridges satisfying $(\text{Br})$, $(\text{Ex})$, $(\text{GP}')$, we have
\begin{equation}\label{eq:expect_faces_bridge_big_sum}
\E\, [f_k(C_n)]
=
2 \sum_{0\leq i_1<\dots<i_{k+1} < n}\frac{Q_{i_1,\dots,i_{k+1}}^{(n)}(d-k-1) + Q_{i_1,\dots,i_{k+1}}^{(n)}(d-k-3)+\dots}{(i_2-i_1)!\dots  (i_{k+1}-i_k)! (n-i_{k+1}+i_1)!}.
\end{equation}
In particular, the expected number of faces is distribution-free both for walks under $(\pm\text{Ex})$, $(\text{GP})$, and for bridges under  $(\text{Br})$, $(\text{Ex})$, $(\text{GP}')$.

In Section~\ref{subsec:simplification} we shall evaluate the sum on the right-hand side of~\eqref{eq:expect_faces_walk_big_sum}, thus proving Theorem~\ref{theo:expected_walk} for \emph{symmetric} random walks satisfying $(\pm\text{Ex})$ and $(\text{GP})$. In order to remove the unnecessary symmetry assumption, we shall prove in Section~\ref{subsec:walks_equals_bridge} that the expected number of $k$-faces of any random walk of length $n$ satisfying assumptions $(\text{Ex})$, $(\text{GP})$ is the same as for any random bridge of length $n+1$ satisfying assumptions $(\text{Br})$, $(\text{Ex})$, $(\text{GP}')$ with $n$ replaced by $n+1$. In particular, the expected number of $k$-faces of a random walk is distribution-free provided $(\text{Ex})$ and $(\text{GP})$ hold. This will show that assumption $(\pm\text{Ex})$ is indeed unnecessary and can be relaxed to $(\text{Ex})$.


\subsection{Proof of Theorem~\ref{theo:expected_walk} in the symmetric case}\label{subsec:simplification}
Let us prove that under assumptions $(\pm\text{Ex})$ and $(\text{GP})$,
\begin{equation}\label{eq:exp_walk_symmetric}
\E [f_k(C_n)]= \frac{2\cdot k!}{n!} \sum_{l=0}^{\infty}\stirling{n+1}{d-2l}  \stirlingsec{d-2l}{k+1}.
\end{equation}

Recall from~\eqref{eq:rising_factorial} that $t^{(j)} = t(t+1)\dots (t+j-1)$ denotes the rising factorial. Let $[t^N] f(t) = \frac 1 {N!} f^{(N)}(0)$ be the coefficient of $t^N$ in the Taylor expansion of a function $f$ around $0$. For $m\in\N_0 = \N \cup \{0\}$ define
\begin{multline*}
R_{n,k}(m) \\
:= [t^m] \sum_{j_0,\dots,j_{k+1}}
\left(\frac{(t+1)(t+3)\dots (t+2j_0-1)}{2^{j_0} j_0!}
\frac{(t+1)(t+3)\dots (t+2j_{k+1}-1)}{2^{j_{k+1}} j_{k+1}!}
\frac{t^{(j_1)}}{t j_1!}\dots \frac{t^{(j_k)}}{t j_k!}\right),
\end{multline*}
where the sum  is taken over all $j_0, j_{k+1}\in\N_0$ and $j_1,\dots,j_k\in\N$ such that $j_0+\dots+j_{k+1} = n$.
With this notation,  Theorem~\ref{1249} (see also~\eqref{eq:expect_faces_walk_big_sum}) implies that
$$
\E [f_k(C_n)]
=
2\sum_{l=0}^{\infty} R_{n,k}(d-k-2l-1).
$$
Thus, to prove~\eqref{eq:exp_walk_symmetric}, it suffices to  show that
\begin{equation}\label{eq:R_n_m}
R_{n,k}(m) =\frac{k!}{n!} \stirlingsec{m+k+1}{k+1} \stirling{n+1}{m+k+1}.
\end{equation}

Expanding the product yields
$$
R_{n,k}(m) = [t^m][x^n] \left(\left(\sum_{j=0}^{\infty}\frac{(t+1)(t+3)\dots (t+2j-1)}{2^{j} j!} x^j\right)^2 \left(\sum_{j=1}^\infty\frac{t^{(j)}}{t j!}x^j \right)^k \right).
$$
Using the binomial series (for $t>0$)
$$
\sum_{j=1}^\infty\frac{t^{(j)}}{t j!}x^j = \frac{(1-x)^{-t} - 1}{t},
\quad
\sum_{j=0}^{\infty}\frac{(t+1)(t+3)\dots (t+2j-1)}{2^{j} j!} x^j = (1-x)^{-\frac 12 (t+1)},
$$
we obtain
$$
R_{n,k}(m)= [t^m][x^n] \left((1-x)^{-t-1} \left(\frac{(1-x)^{-t} - 1}{t}\right)^k \right)
=[x^n][t^m] \left(\eee^{-a(t+1)} \left(\frac{\eee^{-at} - 1}{t}\right)^k \right),
$$
where we introduced the notation $a=a(x) = \log (1-x)$.

Consider the term
$$
\eee^{-a(t+1)} \left(\frac{\eee^{-at} - 1}{t}\right)^k
=
t \eee^{-a}  \left(\frac{\eee^{-at} - 1}{t}\right)^{k+1}
+ \eee^{-a}  \left(\frac{\eee^{-at} - 1}{t}\right)^{k}.
$$
As a consequence of the second equality in~\eqref{eq:stirling_def}, we have
$$
\left( \frac {\eee^{-at}-1}{t}\right)^{k} = \sum_{m=0}^{\infty} (-a)^{m+k} \frac{k!}{(m+k)!} \stirlingsec{m+k}{k} t^m.
$$
Using this formula twice, we obtain
\begin{align*}
[t^m] \eee^{-a(t+1)} \left(\frac{\eee^{-at} - 1}{t}\right)^k
&=
(-a)^{m+k} \frac{1}{(m+k)!} \eee^{-a}  \left(\stirlingsec{m+k}{k+1}(k+1)! +\stirlingsec{m+k}{k}k!\right)\\
&=
(-a)^{m+k} \frac{k!}{(m+k)!} \eee^{-a} \stirlingsec{m+k+1}{k+1},
\end{align*}
where the last line follows from the relation
$$
\stirlingsec{m+k}{k+1}(k+1) +\stirlingsec{m+k}{k} = \stirlingsec{m+k+1}{k+1}.
$$

Now recall that $a= \log (1-x)$ and use the first generating function in~\eqref{eq:stirling_def} to get
\begin{align*}
[x^n]((-a)^{m+k}\eee^{-a})
&=
[x^n] \frac{(-\log (1-x))^{m+k}}{(1-x)}\\
&=
\frac 1{m+k+1} [x^n] \frac{\dd}{\dd x} (-\log (1-x))^{m+k+1}\\
&=
\frac{n+1}{m+k+1} [x^{n+1}](-\log (1-x))^{m+k+1}\\
&=
\frac{(m+k)!}{n!} \stirling{n+1}{m+k+1}.
\end{align*}
Taking everything together, we obtain~\eqref{eq:R_n_m}, thus completing the proof.

\subsection{Relation between random walks and random bridges}\label{subsec:walks_equals_bridge}
The next result completes the proof of Theorem~\ref{theo:expected_walk}.
\begin{theorem}\label{1455}
Let $(S_i)_{i=0}^n$ be a random walk in $\R^d$ whose increments $(\xi_1,\dots,\xi_n)$ satisfy conditions $(\text{Ex})$ and $(\text{GP})$. Further, let $(S_i')_{i=0}^{n+1}$ be a random bridge of length $n+1$  satisfying conditions $(\text{Br})$, $(\text{Ex})$ and $(\text{GP}')$ with $n$ replaced by $n+1$ (in particular, $S_0'=S_{n+1}'=0$). Write $C_n=\conv(S_0,\dots,S_n)$ and $C_{n+1}'= \conv(S_0',\dots,S_{n+1}')$ for the corresponding convex hulls. Then, for all $0\leq k \leq d-1$,
\begin{equation}\label{1613}
\E [f_k(C_n)] =
\E [f_k(C_{n+1}')].
\end{equation}
\end{theorem}

\begin{proof}
From~\eqref{eq:expect_faces_bridge_big_sum} we know that $\E [f_k(C_{n+1}')]$ does not depend on the choice of the particular bridge $(S_i')_{i=0}^{n+1}$. Hence, it suffices to prove~\eqref{1613} for any random bridge of our choice. This bridge will be constructed as follows.
Start with a random walk $(S_i)_{i=0}^n$ satisfying $(\text{Ex})$ and $(\text{GP})$, and consider the closed path $S_0,S_1,\dots, S_n, 0$.  Although this path returns to $S_{n+1}:=0$ at step $n+1$,  its increments are not exchangeable because the last increment $\xi_{n+1}:=-S_n$ and, say, $\xi_1$ have different distributions. In order to enforce the exchangeability, we shall consider a random permutation of the increments of this closed path. More precisely, our construction goes as follows.

For any permutation $\sigma$ from the symmetric group $\Sym(n+1)$, consider the random sequence $(S^{\sigma}_i)_{i=0}^{n+1}$ starting at $S^{\sigma}_0:=0$ and defined by
$$
S^{\sigma}_i:=\xi_{\sigma(1)}+\dots+\xi_{\sigma(i)}, \quad 1\leq i \leq n+1.
$$
Clearly, each sequence terminates at $S^{\sigma}_{n+1} = \xi_1+\dots+\xi_n+\xi_{n+1}=0$. Denote by $C^{\sigma}_{n+1}:=\conv(S^{\sigma}_0,S^{\sigma}_1,\dots,S^{\sigma}_{n+1})$ its convex hull.
Consider a random permutation $\bsigma$ that is uniformly distributed on the symmetric group $\Sym(n+1)$ and independent of the random walk $(S_i)_{i=0}^n$.

It is clear that $(S^{\bsigma}_i)_{i=0}^{n+1}$ satisfies conditions $(\text{Br})$ and $(\text{Ex})$ with $n$ replaced by $n+1$. To verify $(\text{GP}')$ (with $n+1$ as well), by $(\text{Ex})$ it suffices to show that for any $p\in \{1,\ldots,d\}$ and $1 \leq l_1< \ldots < l_p \leq l_{p+1}< \ldots < l_d <n$ and also for $p=0$ and any $0 \leq l_{1}< \ldots < l_d <n$,  the random vectors
$$
S_{l_1}, \ldots, S_{l_p}, S_{l_{p+1}}-S_n, \ldots, S_{l_{d}}-S_n
$$
are linearly independent with probability $1$. Equivalently, the increments
$$
S_{l_1}, S_{l_2}-S_{l_1}, \ldots, S_{l_p} - S_{l_{p-1}}, S_{n} -S_{l_d}, S_{l_d}-S_{l_{d-1}}, \ldots, S_{l_{p+2}}-S_{l_{p+1}}
$$
are a.s.\ linearly independent. Since these increments are taken over disjoint time intervals, their linear independence follows from assumptions  $(\text{Ex})$ and $(\text{GP})$ imposed on $(S_i)_{i=0}^n$, as in the proof of Remark~\ref{rem:simplicial}. This verifies $(\text{GP}')$ (with $n+1$) for $(S^{\bsigma}_i)_{i=0}^{n+1}$.

By the distribution freeness of the expected number of $k$-faces under $(\text{Br})$, $(\text{Ex})$, and $(\text{GP}')$ (see\eqref{eq:expect_faces_bridge_big_sum}), it remains to prove that
\begin{equation}\label{1614}
\E [f_k(C_n)]=\E [f_k(C^{\bsigma}_{n+1})].
\end{equation}
We have
\begin{equation}\label{1838}
\E [f_k(C^{\bsigma}_{n+1})]=\frac{1}{(n+1)!}\sum_{\sigma\in\Sym(n+1)}\E [f_k(\conv(S^\sigma_1,\dots,S^\sigma_{n+1}))].
\end{equation}
Fix a permutation $\sigma\in \Sym(n+1)$ and let $r\in \{1,\dots,n+1\}$ be such that $\sigma(r)=n+1$. Then
\begin{multline*}
  \conv(S^\sigma_1,S^\sigma_2,\dots,S^\sigma_{n+1})=\conv(S^\sigma_{r},\dots,S^\sigma_{n+1},S^\sigma_1,\dots,S^\sigma_{r-1})\\
  =S^\sigma_{r}+\conv(0,S^\sigma_{r+1}-S^\sigma_{r},\dots,S^\sigma_{n+1}-S^\sigma_{r},S^\sigma_{n+1}-S^\sigma_{r}+S^\sigma_1,\dots,S^\sigma_{n+1}-S^\sigma_{r}+S^\sigma_{r-1}),
\end{multline*}
where in the second equality we used that $S^\sigma_{n+1}=0$. Since shifts do not change the number of faces, we arrive at
\begin{multline}\label{1839}
  f_k(\conv(S^\sigma_1,S^\sigma_2,\dots,S^\sigma_{n+1}))\\
  =f_k(\conv(0,S^\sigma_{r+1}-S^\sigma_{r},\dots,S^\sigma_{n+1}-S^\sigma_{r},S^\sigma_{n+1}-S^\sigma_{r}+S^\sigma_1,\dots,S^\sigma_{n+1}-S^\sigma_{r}+S^\sigma_{r-1})).
\end{multline}
It follows from $\sigma(r)=n+1$ and condition $(\text{Ex})$ that
$$
(\xi_{\sigma(r+1)},\dots,\xi_{\sigma(n+1)},\xi_{\sigma(1)},\dots,\xi_{\sigma(r-1)})\eqdistr(\xi_1,\dots,\xi_n),
$$
which implies, by taking partial sums at both sides, that
\begin{equation}\label{1840}
 (0,S^\sigma_{r+1}-S^\sigma_{r},\dots,S^\sigma_{n+1}-S^\sigma_{r},S^\sigma_{n+1}-S^\sigma_{r}+S^\sigma_1,\dots,S^\sigma_{n+1}-S^\sigma_{r}+S^\sigma_{r-1})\eqdistr(0,S_1,\dots,S_n).
\end{equation}
Combining~\eqref{1839} and~\eqref{1840}, we obtain that for every deterministic $\sigma\in\Sym(n+1)$,
$$
\E [f_k(\conv(S^\sigma_1,S^\sigma_2,\dots,S^\sigma_{n+1}))] = \E [f_k(C_n)].
$$
Inserting this into~\eqref{1838} yields~\eqref{1614}, and the theorem follows.
\end{proof}

\section{Proof of Theorem~\ref{theo:1139}}\label{sec:proof_shifted}
The following two propositions combined with Theorem~\ref{1249bridge} yield Theorem~\ref{theo:1139}. Their proofs use the same ideas of reshuffling of increments as in the proof of Theorem~\ref{1455}.
\begin{proposition}\label{prop:shift}
With the same notation and assumptions as in Theorem~\ref{1455}, for all indices $1\leq l_1 < \dots < l_k \leq n$,
\begin{multline*} 
\frac 1 {n+1} \sum_{i=0}^n \P[\conv(S_{i}, S_{i+l_1},\dots, S_{i+l_{k}})\in \cF_k(C_n)]
= \P[\conv(0, S_{l_1}',\dots, S_{l_{k}}')\in \cF_k(C_{n+1}')],
\end{multline*}
where we put $S_{i+l_j} = S_{(i+l_j)-(n+1)}$ if $i+l_j\geq n+1$. 
\end{proposition}
\begin{proof}

The increments $\xi_1,\dots,\xi_n$ of the random walk $(S_i)_{i=0}^n$ satisfy $(\text{Ex})$ and $(\text{GP})$. Repeating the proof of Theorem~\ref{1455}, we define $\xi_{n+1} := -S_n$ and reshuffle the increments $\xi_1,\dots,\xi_{n+1}$ according to a random uniformly distributed permutation $\bsigma\in \Sym(n+1)$ independent of $(S_i)_{i=0}^n$. Since $(S_{i}^\bsigma)_{i=0}^{n+1}$ is a random bridge satisfying $(\text{Br})$, $(\text{Ex})$, and $(\text{GP}')$ (with $n$ replaced by $n+1$), by Theorem~\ref{1249bridge} and the definition of $(S_{i}^\bsigma)_{i=0}^{n+1}$, we have
\begin{multline}\label{eq:shift_faces}
\P[\conv(0, S_{l_1}',\dots, S_{l_{k}}')\in \cF_k(C_{n+1}')]
\\=
\frac 1 {(n+1)!} \sum_{\sigma\in \Sym(n+1)}
\P[\conv (S_0^{\sigma}, S_{l_1}^\sigma,\dots, S_{l_{k}}^\sigma) \in \cF_k(\conv (S_0^{\sigma},S_1^{\sigma}, \dots, S_{n+1}^{\sigma}))].
\end{multline}

Fix a permutation $\sigma\in \Sym(n+1)$ and let $r\in \{1,\dots,n+1\}$ be such that $\sigma(r) = n+1$. Since the shift by $S_r^{\sigma}$ does not change the structure of the convex hull, we have
\begin{multline}\label{eq:wspom1}
\P[\conv (S_0^{\sigma}, S_{l_1}^\sigma,\dots, S_{l_{k}}^\sigma) \in \cF_k(\conv (S_0^{\sigma},S_1^{\sigma}, \dots, S_{n+1}^{\sigma}))]
\\=
\P[\conv (S_0^{\sigma}-S_r^{\sigma}, S_{l_1}^\sigma-S_r^{\sigma},\dots, S_{l_{k}}^\sigma-S_r^{\sigma})
\in \cF_k(\conv (S_0^{\sigma}-S_r^{\sigma},S_1^{\sigma}-S_r^{\sigma}, \dots, S_{n}^{\sigma}-S_r^{\sigma}))].
\end{multline}
Recall that from~\eqref{1840} and $S^\sigma_{n+1} = 0$,
\begin{equation}\label{1840a}
(S_r^{\sigma}- S_r^{\sigma},S^\sigma_{r+1}-S^\sigma_{r},\dots,S^\sigma_{n+1}-S^\sigma_{r},S^\sigma_1-S^\sigma_{r},\dots,S^\sigma_{r-1}-S^\sigma_{r})
\eqdistr(0,S_1,\dots,S_n).
\end{equation}
Note that $S_l^\sigma-S_r^{\sigma}$ on the left-hand side corresponds to $S_{l-r}$ on the right-hand side if we agree to understand all indices modulo $n+1$.   Applying~\eqref{1840a} to the right-hand side of~\eqref{eq:wspom1} and using the fact that $\{S_l\}_{l=0}^n = \{S_{l-r}\}_{l=0}^n$,  we arrive at
\begin{multline*}
\P[\conv (S_0^{\sigma}, S_{l_1}^\sigma,\dots, S_{l_{k}}^\sigma) \in \cF_k(\conv (S_0^{\sigma},S_1^{\sigma}, \dots, S_{n+1}^{\sigma}))]
\\=
\P[\conv (S_{-r}, S_{l_1-r},\dots, S_{l_{k}-r})
\in \cF_k(\conv (S_0,S_1, \dots, S_{n}))].
\end{multline*}
Taking the sum over all $\sigma \in \Sym(n+1)$ and observing that for any fixed $r\in \{1,\dots,n+1\}$ there are $n!$ permutations $\sigma$ for which $\sigma(r) = n+1$, we arrive at
\begin{multline*}
\text{RHS of~\eqref{eq:shift_faces}}
= \frac{n!}{(n+1)!} \sum_{r=1}^{n+1} \P[\conv (S_{-r}, S_{l_1-r},\dots, S_{l_{k}-r})
\in \cF_k(\conv (S_0,S_1, \dots, S_{n}))].
\end{multline*}
Substituting  $i = n+1-r$ and recalling that the indices are considered modulo $n+1$ completes the proof.
\end{proof}

\begin{proposition}\label{1139}
With the same notation and assumptions as in Theorem~\ref{1455}, for all indices $1\leq l_1 < \dots < l_k \leq n$,
\begin{multline} \label{1921}
\frac 1 {n-l_k+1} \sum_{i=0}^{n-l_k} \P[\conv(S_{i}, S_{i+l_1},\dots, S_{i+l_{k}})\in \cF_k(C_n)]
=\P[\conv(0, S_{l_1}',\dots, S_{l_{k}}')\in \cF_k(C_{n+1}')].
\end{multline}
\end{proposition}
\begin{proof}
The main idea is to combine the proof of Theorem~\ref{1249bridge} with the method of reshuffling from the proof of Theorem~\ref{1455}.
From Theorem~\ref{1249bridge} we know that the face probability on the right-hand side of~\eqref{1921} is distribution-free under $(\text{Br})$, $(\text{Ex})$, and $(\text{GP}')$. Hence, it suffices to prove~\eqref{1921} for a bridge of our choice.

The increments $\xi_1,\dots,\xi_n$ of the random walk $(S_i)_{i=0}^n$ satisfy $(\text{Ex})$ and $(\text{GP})$. Similarly to the proof of Theorem~\ref{1455}, we define $\xi_{n+1} := -S_n$ and reshuffle the increments $\xi_1,\dots,\xi_{n+1}$ according to a random permutation $\bsigma\in \Sym(n+1)$ given by
$$\bsigma =(1, \ldots, l_k, l_k + \bsigma'(1), \ldots, l_k + \bsigma'(n-l_k +1)),$$
where $\bsigma' \in \Sym(n-l_k+1)$ is a uniformly distributed random permutation independent of $(S_i)_{i=0}^n$.

Let $\rho$ be the random variable defined by $\bsigma'(\rho) = n-l_k+1$. Clearly, $\rho$ is distributed uniformly on $\{1, \ldots, n-l_k+1 \}$. Consider the random permutation $\tau \in \Sym(n+1)$
$$\tau:=(l_k+ \bsigma'(\rho+ 1), \dots, l_k+\bsigma'(n-l_k+1), 1,  \ldots, l_k, l_k + \bsigma'(1), \ldots,l_k+ \bsigma'(\rho-1), n+1)$$
obtained by a cyclic shift of $\bsigma$, that is $\tau(n-l_k-\rho+1+i) = \bsigma(i)$ for $1 \le i \le n+1$ if we agree to understand all indices modulo $n+1$. Then
$$(S_{i}^\bsigma)_{i=0}^{n} = -S_{n-l_k-\rho+1}^\tau    + (S_{n-l_k-\rho+1}^\tau , S_{n-l_k-\rho+2}^\tau , \ldots, S_n^\tau, 0, S_1^\tau, \ldots, S_{n-l_k-\rho}^\tau),$$
implying, by the equality
$$\{S_{n-l_k-\rho+1}^\tau , S_{n-l_k-\rho+2}^\tau , \ldots, S_n^\tau, 0, S_1^\tau, \ldots, S_{n-l_k-\rho}^\tau \} = \{ S_i \}_{i=0}^n$$
and the fact that shifts do not change the structure of the convex hull, the equality of the events
\begin{multline*}
\bigl \{ \conv(0, S^\bsigma_{l_1},\ldots, S^\bsigma_{l_k})\in \cF_k(C^\bsigma_{n+1}) \bigr \}\\
= \bigl \{  \conv(S_{n-l_k-\rho+1}^\tau, S_{(n-l_k-\rho+ 1) + l_1}^\tau, \ldots, S_{(n - l_k -\rho + 1) + l_k}^\tau) \in \cF_k(C_n) \bigr \}.
\end{multline*}
Finally, we pass to probabilities and condition on $\rho$ in the right-hand side. Using the distributional identity $\text{Law}((S_i)_{i=0}^n ) = \text{Law} ((S_i^\tau)_{i=0}^n | \rho = \rho_0)$ for all $\rho_0\in \{1,\ldots, n-l_k+1\}$, which holds since
the random walk $(S_i)_{i=0}^n$ satisfies  $(\text{Ex})$ and the increments of $(S_i^\tau)_{i=0}^n $ do not include $\xi_{n+1}$, we arrive at
\begin{equation} \label{eq: conv permuted}
\P[\conv(0, S^\bsigma_{l_1},\ldots, S^\bsigma_{l_k})\in \cF_k(C^\bsigma_{n+1})]
=
\frac 1 {n-l_k+1} \sum_{i=0}^{n-l_k} \P[\conv (S_i, S_{i+l_1},\dots, S_{i+l_{k}}) \in \cF_k(C_n)].
\end{equation}

It remains to compute the left-hand side to prove the proposition. We cannot apply Theorem~\ref{1249bridge} directly since $(S_{i}^\bsigma)_{i=0}^{n+1}$ does not have exchangeable increments: $\xi_{\bsigma(n+1)}$ and, say, $\xi_1$ clearly have different distributions. However, the argument of the proof of Theorem~\ref{1249bridge} applies directly, and we repeat it below.

Consider the linear hull $M:=\lin^\perp(S_0^\bsigma,S_{l_1}^\bsigma,\dots,S_{l_{k}}^\bsigma)$ and note that $\dim M = d-k$ a.s.  Projecting the path $S_0^\bsigma, \dots,S_{n+1}^\bsigma$ on $M$, we obtain $k+1$ random bridges that  start and terminate at $0 \in M$. The increments of the bridges are given by
$$
\eta_1^{(j)}=\xi_{l_{j-1}+1}|M, \;\; \dots, \;\; \eta_{l_{j}-l_{j-1}}^{(j)}=\xi_{l_{j}}|M,
\quad j=1,\dots k,
$$
where $l_0:=0$, and
$$
\eta_1^{(k+1)}=\xi_{l_k+\bsigma'(1)}|M, \;\; \dots, \;\; \eta_{n-l_{k}+1}^{(k+1)}=\xi_{l_k+\bsigma'(n-l_k+1)}|M.
$$
Let $H_0^\bsigma:=C^\bsigma_{n+1}|M$ denote the convex hull of these $k+1$ random bridges in $M \cong \R^{d-k}$ a.s. Then
\begin{equation} \label{eq: H permuted}
\P[\conv(0, S^\bsigma_{l_1},\ldots, S^\bsigma_{l_k})\in \cF_k(C^\bsigma_{n+1})] = \P[0 \in \mathcal{F}_0(H_0^\bsigma)].
\end{equation}

It is easy to see that  the invariance condition~\eqref{eq:invar_product} of Theorem~\ref{1435} is satisfied for these $r=k+1$ random bridges (and $s=0$ random walks). The argument uses the same ideas as in the proofs of Theorems~\ref{1249} and~\ref{1249bridge}. The general position assumption of Theorem~\ref{1435} is verified in the same manner as in the  proof of Theorem~\ref{1249bridge} supplemented by the respective argument from the proof of Theorem~\ref{1455}.

Now 
\eqref{1921} follows by combining \eqref{eq: conv permuted} with \eqref{eq: H permuted} and applying Theorem~\ref{1435} (see also Remark~\ref{rem:zero included} and~\eqref{eq:non_absorption}) to $H_0^\bsigma$ exactly as in the proof of Theorem~\ref{1249bridge}.
\end{proof}

\section{Further remarks and conjectures}

\subsection{Results not requiring the general position assumption}
Let $(S_i)_{i=0}^n$ and $(S_i')_{i=0}^n$ be two random walks with increments $(\xi_1,\dots, \xi_n)$ and $(\xi_1',\dots, \xi_n')$, respectively, such that $(\xi_1,\dots, \xi_n)$ satisfies $(\pm\text{Ex})$ and $(\text{GP})$, while $(\xi_1',\dots, \xi_n')$ satisfies $(\pm\text{Ex})$ only.  Denote the corresponding convex hulls by $C_n := \conv(S_0,\dots,S_n)$ and $C_n' := \conv(S_0',\dots,S_n')$. All faces of $C_n$ are simplices a.s., see Remark~\ref{rem:simplicial}, but this is in general not true for $C_n'$. The face probabilities of $C_n'$ are not distribution-free, but it can be shown that
\begin{multline}\label{eq:no_general_pos_inequality}
\P[\conv (S_{i_1}', \dots, S_{i_{k+1}}') \in \mathcal F_k(C_n')]
\leq
\P[\conv (S_{i_1}, \dots, S_{i_{k+1}}) \in \mathcal F_k(C_n)]\\
\leq
\P[\conv (S_{i_1}', \dots, S_{i_{k+1}}')\subset F' \text{ for some } F' \in \mathcal F_k(C_n')],
\end{multline}
where the distribution-free  probability in the middle was calculated in Theorem~\ref{1249}. To prove this, one argues in the same way as in the proof of Theorem~\ref{1249}, but uses Remark~\ref{rem:non_gen_position_absorption} instead of Theorem~\ref{1435}. Similar inequalities hold for bridges violating the general position assumption.

From~\eqref{eq:no_general_pos_inequality} it is possible to deduce that under assumption $(\text{Ex})$, one has
$$
\E [f^{-}_k (C_n')] \leq  \frac{2\cdot k!}{n!} \sum_{l=0}^{\infty}\stirling{n+1}{d-2l}  \stirlingsec{d-2l}{k+1}
\leq
\E [f^{+}_k (C_n')],
$$
where $f^{-}_k(C_n')$ is the number of simplicial $k$-faces of $C_n'$, while $f^{+}_k(C_n')$ is the number of collections of indices $1\leq i_1 < \dots < i_{k+1} \leq n$ such that $\conv(S_{i_1}, \dots, S_{i_{k+1}})$ is contained in some $k$-face $F'$ of $C_n'$.



\subsection{Expected total number of faces}
Assuming  $(\text{Ex})$ and $(\text{GP})$, let us compute the expected number of faces of $C_n$ in all dimensions together.
We need the numbers $\widehat  c_N$ (which appeared in~\cite{sprugnoli}) defined by
\begin{equation}\label{eq:def_hat_c_N}
\widehat  c_N = \sum_{k=1}^N (k-1)!  \stirlingsec  {N}{k} = N! [t^N] \log \left(\frac 1 {2-\eee^t}\right) = (N-1)! [t^{N-1}] \left(\frac{\eee^t}{2-\eee^t}\right).
\end{equation}
From Theorem~\ref{theo:expected_walk} and~\eqref{eq:stirling_asympt} we obtain the formula
$$
\E \left[\sum_{k=0}^{d-1}f_{k}(C_n)\right]
=
\frac 2 {n!}\sum_{l=0}^\infty  \stirling{n+1}{d-2l} \widehat  c_{d-2l}
\sim  \frac{2 \widehat  c_{d}}{(d-1)!} (\log n)^{d-1},
$$
where the asymptotics is for $n\to\infty$ and fixed $d\in\N$.
The numbers $\widehat  c_N$ are similar to the \textit{ordered Bell numbers} $\mathcal O_N$ which count the number of weak orderings on a set of $N$ elements and are given by
$$
\mathcal O_N = \sum_{k=0}^N k! \stirlingsec{N}{k} = \frac 12 \sum_{i=0}^\infty \frac{i^N}{2^i}.
$$

\subsection{Open questions}
It is natural to ask for the limit distribution of the (appropriately normalized) number of $k$-faces of $C_n$. For convex hulls of i.i.d.\ Gaussian samples, variance asymptotics for the face numbers were established recently by Calka and Yukich~\cite{calka_yukich} (also see earlier works~\cite{BV07, Hueter, HR05} for limit theorems as LLN and CLT), but the i.i.d.\ model may behave very differently from the convex hulls of random walks studied in the present paper.  Example~\ref{ex:non_symm_RW} shows that the absorption probability is not distribution-free for (non-symmetric) random walks satisfying $(\text{Ex})$ and $(\text{GP})$. It is likely that among such random walks the absorption probability attains its maximum value for random walks satisfying $(\pm\text{Ex})$.  A similar result for convex hulls of i.i.d.\ samples was proved by Wagner and Welzl~\cite{wagner_welzl}.

It also remains open to compute the expected number of faces and the face probabilities for the \emph{simple} random walk on the lattice $\Z^d$, even for $d=2$.
\section*{Acknowledgments}
We would like to thank the referee for his/her  stimulating comments and suggestions.

\bibliography{k-face-bib}
\bibliographystyle{plain}

\end{document}